\documentclass[11pt,reqno]{amsart}
\usepackage{amssymb,mathrsfs,bm,enumerate}
\usepackage{amsfonts,dsfont,
mathptmx,amsaddr}

\makeatletter

\newtheorem{theorem}{Theorem}[section]
\newtheorem{lemma}[theorem]{Lemma}
\newtheorem{corollary}[theorem]{Corollary}

\theoremstyle{definition}
\newtheorem{remark}[theorem]{Remark}

\numberwithin{equation}{section}

\def\Ric{{\operatorname{Ric}}} \def\Aut{{\operatorname{Aut}}}
\def\End{{\operatorname{End}}} \def\Hess{{\operatorname{Hess}}}
\def\II{{\operatorname{II}}} \def\II{{\operatorname{II}}}

\newcommand\1{\hbox{\kern.375em\vrule height1.57ex depth-.1ex
    width.05em\kern-.375em \rm 1}}

 \newcommand\E{\mathbb{E}}
 
 \newcommand\R{\mathbb{R}}
\renewcommand\P{\mathbb{P}} 
 \newcommand\varK{\mathbb{K}}

\def\mathpal#1{\mathop{\mathchoice{\text{\rm #1}}%
    {\text{\rm #1}}{\text{\rm #1}}%
    {\text{\rm #1}}}\nolimits} \def\id{{\mathpal{id}}}
\def\OM{\mathpal{O}(M)} \def\OtM{{\mathpal{O}}_t(M)}
\def\FM{\mathpal{F}(M)} \def\FM{F(M)} \def\Gl{\mathpal{GL}}

\def\partr#1#2{/\!/_{\!#1,#2}^{\phantom{.}}}
\def\partrinv#1#2{/\!/_{\!#1,#2}^{-1}}

 \def\vd{\mathrm{d}} \def\r{\right}
\def\l{\left} \def\e{\operatorname{e}} \def\dsum{\displaystyle\sum}
  \def\dlim{\displaystyle\lim}

  \makeatother

\begin{document}

  \title[Characterization of pinched Ricci curvature]
  {Characterization of pinched Ricci curvature\\ by functional
    inequalities}

  \author{Li-Juan Cheng\textsuperscript{1,2} and Anton
    Thalmaier\textsuperscript{1}}

  \address{\textsuperscript{1}Mathematics Research Unit, FSTC, University of Luxembourg\\
    6, rue Richard Coudenhove-Kalergi, 1359 Luxembourg, Grand Duchy of
    Luxembourg}
  \address{\textsuperscript{2}Department of Applied Mathematics, Zhejiang University of Technology\\
    Hangzhou 310023, The People's Republic of China}
  \email{lijuan.cheng@uni.lu \text{\rm and} chenglj@zjut.edu.cn}
  \email{anton.thalmaier@uni.lu}

\begin{abstract}
  In this article, functional inequalities for diffusion semigroups on
  Riemannian manifolds (possibly with boundary) are established, which
  are equivalent to pinched Ricci curvature, along with gradient
  estimates, $L^p$-inequalities and log-Sobolev inequalities. These
  results are further extended to differential manifolds carrying
  geometric flows. As application, it is shown that they can be used
  in particular to characterize general geometric flow and Ricci flow by
  functional inequalities.
\end{abstract}

\date{\today}

\maketitle

\section{Introduction}\label{Sect:1}
Let $(M,g)$ be a $d$-dimensional Riemannian manifold, possibly with
boundary. Let $\nabla$ and $\Delta$ be the Levi-Civita connection and
the Laplacian associated with the Riemannian metric $g$,
respectively. For a given $C^{1}$-vector field~$Z$ on $M$ and tangent
vectors $X, Y$ on $M$, let
\begin{align*}
  \Ric^{Z}(X,Y):=\Ric(X,Y)-\l<\nabla_XZ, Y\r>,
\end{align*}
where $\Ric$ is the Ricci curvature tensor with respect to $g$ and
$\l<\cdot,\cdot\r>:=g(\cdot,\cdot)$.  We denote by $C(M)$, $C_b(M)$,
$C^\infty(M)$ and $C_0^\infty(M)$ the sets of continuous functions, 
bounded continuous functions, smooth functions, 
smooth test functions on $M$, respectively.

Given a $C^{1}$-vector field $Z$ on $M$, we consider the elliptic
operator $L:=\Delta+Z$.  Let $X_t^x$ be a diffusion process starting
from $X^x_0=x$ with generator $L$, called a $L$-diffusion process.  We
assume that $X_t^x$ is non-explosive for each $x\in M$.  Let
$B_t=(B_t^1,\ldots,B_t^d)$ be a $\R^d$-valued Brownian motion on a
complete filtered probability space
$(\Omega, \{\mathscr{F}_t\}_{t\geq 0},\P)$ with the natural filtration
$\{\mathscr{F}_t\}_{t\geq 0}$. The $L$-diffusion process $X^x_t$
starting from $x$ can be constructed as a solution to the Stratonovich
equation
\begin{align}\label{SDE-X}
  \vd X^x_t=\sqrt{2}u^x_t\circ \vd B_t+Z(X^x_t)\,
  \vd t,\quad X^x_0=x,
\end{align}
where $u^x_t$ is the horizontal process of $X^x_t$ taking values in
the orthonormal frame bundle $\OM$ over $M$ such that
$\pi(u^x_0)=x$. Note that
$$\partr st:=u^x_t\circ (u^x_s)^{-1}\colon{T_{X^x_s}M}\to{T_{X^x_t}M},\quad s\leq t,$$
defines parallel transport along the paths $r\mapsto X^x_r$.  By
convention, an orthonormal frame $u\in\OM$ is interpreted as isometry
$u\colon\R^d\to T_xM$ where $\pi(u)=x$.  Note that parallel transport
$\partr st$ is independent of the choice of the initial frame $u^x_0$
above $x$.

The diffusion process $X^x_t$ gives rise to a Markov semigroup $P_t$
with infinitesimal generator $L$: for $f\in C_b(M)$, we have
$$P_tf(x)=\E [f(X_t^x)],\quad t\geq0,$$
where $\E$ stands for expectation with respect to the underlying
probability measure $\P$.

The problem of characterizing boundedness of $\Ric^Z$ from below in
terms of gradient estimates and other functional inequalities for the
semigroup $P_t$, has been thoroughly studied in the literature,
e.g.~\cite{H02,Wbook1,Wbook2}.  For instance, it is well-known that
the curvature condition
$$\Ric^Z(X,X)\geq \kappa|X|^2, \quad X\in TM,$$
is equivalent to each of the following inequalities:
\begin{enumerate}[ 1)]
\item (gradient estimate) for all $f\in C_0^{\infty}(M)$,
  $$|\nabla P_tf|^2\leq \e^{-2\kappa t}P_t|\nabla f|^2;$$
\item (Poincar\'e inequality) for all $p\in (1,2]$ and
  $f\in C_0^{\infty}(M)$,
  $$\frac{p}{4(p-1)}(P_tf^2-(P_tf^{2/p})^p)\leq \frac{1-\e^{-2\kappa
      t}}{2\kappa}P_t|\nabla f|^2;$$
\item (log-Sobolev inequality) for all $f\in C_0^{\infty}(M)$,
  $$P_t(f^2\log f^2)-P_tf^2\log P_tf^2\leq \frac{2(1-\e^{-2\kappa
      t})}{\kappa}P_t|\nabla f|^2.$$
\end{enumerate}
However, the question how to use functional inequalities for $P_t$
to characterize upper bounds on $\Ric^Z$ is much more delicate.  When
it comes to stochastic analysis on path space, there is a lot of
former work based on bounds of $\Ric^Z$, see
e.g. \cite{CHL,CW,Dri92,Hsu94}.  Recently, A.~Naber \cite{Naber} and
R.~Haslhofer and A.~Naber \cite{HN1} have been able to establish
gradient inequalities on path space which characterize boundedness of $\Ric^Z$;
F.-Y.~Wang and B.~Wu \cite{WW} extended these results to manifolds
with boundary, where $\Ric^Z$ may also vary along the manifold and may
be unbounded.

Let us briefly describe R.~Haslhofer and A.~Naber's work. Among other
things, they prove that the functional inequality,
\begin{equation}\label{Eq:Naber-Haslhofer}
  |\nabla \E F(X_{[0,T]})|^2\leq \e^{\kappa T}\E\l[|D_0^{/\!/}F|^2
  +\kappa \int_0^T \e^{\kappa(r-T)}|D_{r}^{/\!/}F|^2\,\vd r\r],
  \quad F\in \mathcal{F}C_0^{\infty},
\end{equation}
is equivalent to the curvature condition $|\Ric^Z|\leq \kappa$ for
some nonnegative constant $\kappa$, where
$$\mathcal{F}C_0^{\infty}:=\left\{f(X_{t_1},\ldots, X_{t_N})\colon\
  0\leq t_1<\ldots<t_N\leq T,\ f\in C_0^{\infty}(M^{N})\right\}$$ and
$$D_t^{/\!/}F(X_{[0,T]}):=\sum_{i=1}^{N}\1_{\{t\leq t_i\}}\,\partrinv t{t_i}\nabla_iF(X_{[0,T]}),
\quad F\in \mathcal{F}C_0^{\infty}.$$
In their proof, in order to show that gradient estimate
\eqref{Eq:Naber-Haslhofer} above implies $|\Ric^Z|\leq \kappa$, they
show that it is sufficient to consider 2-point cylindrical functions
of the special type
$$F(X_{[0,T]})=f(x)-\frac{1}{2}f(X_t)$$
as test functional.  From this observation, it is easy to see that the
subsequent items (i) and (ii) are equivalent:
\begin{itemize}
\item [(i)] $|\Ric^Z|\leq \kappa$ for $\kappa\geq 0$;
\item [(ii)] for $f\in C_0^{\infty}(M)$ and $t>0$,
  \begin{align*}
    & |\nabla P_tf|^2\leq \e^{2\kappa t}P_t|\nabla f|^2\quad\text{and}\\
    & \l|\nabla f(x)-\frac{1}{2}\nabla P_{t}f\r|^2\leq \e^{\kappa t}\E\l[\l|\nabla f-\frac{1}{2}\partrinv0t\nabla f(X_t)\r|^2+\frac{1}{4}\left(\e^{\kappa t}-1\right)|\nabla f(X_t)|^2\r].
  \end{align*}
\end{itemize}
Note that the inequalities in (ii) can be combined to the single
condition:
$$|\nabla P_tf|^2-\e^{2\kappa t}P_t|\nabla f|^2\leq 4\l((\e^{\kappa t}-1)|\nabla f|^2+\l<\nabla f, \nabla P_tf\r>-\big<\nabla f, \e^{\kappa t}\E[ \partrinv0t\nabla f(X_t)]\big>\r)\wedge 0.$$

The discussion above gives rise to a natural question: \emph{Are there
  gradient inequalities on $M$ which allow to characterize pinched
  curvature with arbitrary upper and lower bounds?}

Our paper is organized as follows. In Section~\ref{Sect:2} we give a
positive answer to the question above.  In Section~\ref{Sect:3}, we
extend these results to characterize simultaneous bounds on $\Ric^Z$
and $\II$ on Riemannian manifolds with boundary, where the curvature
bounds are not given by constants, but may vary over the manifold.  In
Section~\ref{Sect:4} finally, we present gradient and functional
inequalities for the time-inhomogeneous semigroup $P_{s,t}$ on
manifolds carrying a geometric flow.  We show that these inequalities
can be used to characterize solutions to some geometric flows,
including Ricci flow.

\section{Characterizations for Ricci curvature}\label{Sect:2}
We start the section by introducing our main results.

\begin{theorem}\label{th1}
  Let $(M,g)$ be a complete Riemannian manifold. Let $k_1, k_2$ be two
  real constants such that $k_1\leq k_2$.  The following conditions
  are equivalent:
  \begin{enumerate}[\rm(i)]
  \item[\rm(i)] $k_1\leq \Ric^Z\leq k_2$;
  \item[\rm(ii)] for $f\in C_0^{\infty}(M)$ and $t>0$,
 $$\quad|\nabla P_tf|^2-\e^{-2k_1t}P_t|\nabla f|^2\leq 4\l[\left(\e^{\frac{k_2-k_1}{2}t}-1\right)|\nabla f|^2+\l<\nabla f, \nabla P_tf\r>-\e^{-k_1t}\E\big<\nabla f, \partrinv0t\nabla f(X_t)\big>\r]\wedge 0;$$
\item [\rm(ii')] for $f\in C_0^{\infty}(M)$ and $t>0$,
 $$|\nabla P_tf|^2-\e^{-2k_1t}P_t|\nabla f|^2\leq 4\l(\e^{\frac{k_2-k_1}{2}t}|\nabla P_tf|^2-\e^{-k_1t}\E\big<\nabla P_tf, \partrinv0t\nabla f(X_t)\big>\r)\wedge 0;$$
\item [\rm(iii)] for $f\in C_0^{\infty}(M)$, $p\in (1,2]$ and $t>0$,
  \begin{align*}
    & \frac{p(P_tf^2-(P_tf^{2/p})^p)}{4(p-1)}-\frac{1-\e^{-2k_1t}}{2k_1}P_t|\nabla f|^2\\
    &\quad\leq 4\int_0^t\l(\e^{\frac{k_2-k_1}{2}(t-r)}-1\r)P_r|\nabla f|^2+\E\big<\nabla f(X_r),\nabla P_{t-r}f(X_r)-\e^{-k_1(t-r)}\partrinv rt\nabla f(X_t)\big>\,\vd r\wedge 0;
  \end{align*}
\item [\rm(iii')] for $f\in C_0^{\infty}(M)$, $p\in (1,2]$ and $t>0$,
  \begin{align*}
    & \frac{p(P_tf^2-(P_tf^{2/p})^p)}{4(p-1)}-\frac{1-\e^{-2k_1t}}{2k_1}P_t|\nabla f|^2\\
    &\quad\leq 4\int_0^t\e^{\frac{k_2-k_1}{2}(t-r)}P_r|\nabla P_{t-r}f|^2-\e^{-k_1(t-r)}\E\l<\nabla f(X_r),\partrinv rt\nabla f(X_t)\r>\,\vd r\wedge 0;
  \end{align*}
\item [\rm(iv)] for $f\in C_0^{\infty}(M)$ and $t>0$,
  \begin{align*}
    & \frac{1}{4}\left(P_t(f^2\log f^2)-P_tf^2\log P_tf^2\right)-\frac{1-\e^{-2k_1t}}{2k_1}P_t|\nabla f|^2\\
    &\quad\leq 4\int_0^t\left(\e^{\frac{k_2-k_1}{2}(t-r)}-1\right)P_r|\nabla f|^2+\E\big<\nabla f(X_r),\nabla P_{t-r}f(X_r)-\e^{-k_1(t-r)}\partrinv rt\nabla f(X_t)\big>\,\vd r\wedge 0;
  \end{align*}
\item [\rm(iv')] for $f\in C_0^{\infty}(M)$ and $t>0$,
  \begin{align*}
    & \frac{1}{4}\left(P_t(f^2\log f^2)-P_tf^2\log P_tf^2\right)-\frac{1-\e^{-2k_1t}}{2k_1}P_t|\nabla f|^2\\
    &\quad\leq 4\int_0^t\e^{\frac{k_2-k_1}{2}(t-r)}P_r|\nabla P_{t-r}f|^2-\e^{-k_1(t-r)}\E\l<\nabla f(X_r),\partrinv rt\nabla f(X_t)\r>\,\vd r\wedge 0.
  \end{align*}
\end{enumerate}
\end{theorem}
\begin{remark}\label{rem1}
  The inequalities in (iv) and (iv') can be understood as limits of
  the inequalities (iii) and (iii') as $p\downarrow 1$ respectively.
\end{remark}

\begin{remark}\label{lem4}
  As application, Theorem~\ref{th1} can be used to characterize
  Einstein manifolds where $\Ric$ is a multiple of the metric $g$
  (constant Ricci curvature). The case $\Ric=\nabla Z$ can be
  characterized by all/some of the inequalities in (ii)-(iv) and
  (ii')-(iv') for $k_1=k_2=0$, where the inequalities in (iii),
  (iii'), (iv) and (iv') may be understood as $k_2=k_1$ and
  $k_1\rightarrow 0$.
\end{remark}

\begin{proof}[Proof of Theorem \ref{th1}.]
  We divide the proof into two parts. In Part I, we will derive the
  functional inequalities from the curvature condition; in Part II, we
  will prove the reverse.\smallskip

  {\bf Part I}. We already know that the curvature condition
  $\Ric^Z\geq k_1$ is equivalent to each of the following functional
  inequalities (see e.g. \cite[Theorem 2.3.1]{Wbook2}):
  \begin{enumerate}[1)]
  \item for all $f\in C_0^{\infty}(M)$,
    $$|\nabla P_tf|^2\leq \e^{-2k_1t}P_t|\nabla f|^2;$$
  \item for all $p\in (1,2]$ and $f\in C_0^{\infty}(M)$,
    $$\frac{p}{4(p-1)}\left(P_tf^2-(P_tf^{2/p})^p\right)\leq
    \frac{1-\e^{-2k_1t}}{2k_1}P_t|\nabla f|^2;$$
  \item for all $f\in C_0^{\infty}(M)$,
    $$P_t(f^2\log f^2)-P_tf^2\log P_tf^2\leq
    \frac{2(1-\e^{-2k_1t})}{k_1}P_t|\nabla f|^2.$$
  \end{enumerate}
  Now, we prove that under the curvature condition (i) in Theorem
  \ref{th1}, the remaining bounds in (ii)-(iv) and (ii')-(iv') hold
  true.\smallskip

  {\bf(a) }\ $\text{(i)}\Rightarrow\text{(ii)}$ and (ii'): We start
  with well-known stochastic representation formulas for diffusion
  semigroups. By Bismut's formula (see \cite{Bismut, EL}), we have
$$(\nabla P_tf)(x)=\E[Q_t\partrinv0t \nabla f(X_t^x)].$$
Here $Q_t$ is the $\Aut(T_xM)$-valued process defined by the linear
pathwise differential equation
\begin{equation}\label{Eq:Qt}
  \frac{\vd}{\vd t}Q_t=-Q_t\Ric^Z_{\partr 0t},\quad Q_0=\id_{T_xM},
\end{equation}
where
\begin{equation}\label{Eq:RicZ0t}
  \Ric^Z_{\partr 0t}:=\partrinv 0t\circ\Ric^Z_{X_t}\circ\partr
  0t\in\End(T_xM)
\end{equation}
and $\partr0t$ is parallel transport in $TM$ along~$X_t$. As usual,
$\Ric^Z_x$ operates as a linear homomorphism on $T_xM$ via
$\Ric^Z_xv=\Ric^Z(\cdot,v)^\sharp$, $v\in T_xM$.


Let $a$ and $b$ be two constants such that $a+b=1$. We first observe
that
\begin{align*}
  &2a\nabla f-2b \nabla P_tf -Q_t\partrinv0t\nabla f(X_t)\\
  &=2a\nabla f-2b \nabla P_tf-\e^{-\frac{k_2+k_1}{2}t}\partrinv0t\nabla f(X_t)
    +\e^{-\frac{k_2+k_1}{2}t}\,\left(\id-\e^{\frac{k_2+k_1}{2}t}Q_t\right)\partrinv0t\nabla f(X_t)
\end{align*}
which implies that
\begin{align}\label{add-1}
  &\left|2(a\nabla f+b\nabla P_tf)-Q_t\partrinv0t\nabla f(X_t)\right|\notag\\
  &\leq \left|2(a\nabla f+b\nabla P_tf)-\e^{-\frac{k_2+k_1}{2}t}\partrinv0t\nabla f(X_t)\right|
    +\left|\e^{-\frac{k_2+k_1}{2}t}\,\left(\id-\e^{\frac{k_2+k_1}{2}t}Q_t\right)\partrinv0t\nabla f(X_t)\right|.
\end{align}
We now turn to estimate the last term on the right-hand side above,
\begin{align*}
  \left|\left(\id-\e^{\frac{k_2+k_1}{2}t}Q_t\right)\partrinv0t\nabla f(X_t)\right|\leq \left\|\id-\e^{\frac{k_2+k_1}{2}t}Q_t\right\|\,|\nabla f(X_t)|.
\end{align*}
To estimate $\|\id-\e^{\frac{k_2+k_1}{2}t}Q_t\|$, we rewrite the
involved operator as
$$\id-\e^{\frac{k_2+k_1}{2}t}Q_t=\int_0^t \e^{\frac{k_2+k_1}{2}s}Q_{s}\l(\Ric^Z_{\partr0s}-\frac{k_1+k_2}{2}\id\r)\vd s.$$
Hence, by the curvature condition (i), we have
\begin{align*}
  \left\|\id-\e^{\frac{k_2+k_1}{2}t}Q_t\right\|
  &\leq \int_0^t \e^{\frac{k_2+k_1}{2}s}\|Q_{s}\|\l|\Ric^Z_{\partr0s}-\frac{k_1+k_2}{2}\,\id\r|\,\vd s\\
  &\leq \int_0^t \e^{\frac{k_2+k_1}{2}s} \e^{-k_1s}\frac{k_2-k_1}{2}\,\vd s
    = \e^{\frac{(k_2-k_1)t}{2}}-1
\end{align*}
which implies
$$\left|\e^{-\frac{k_2+k_1}{2}t}\left(\id-\e^{\frac{k_2-k_1}{2}t}Q_t\right)\partrinv0t\nabla f(X_t)\right|\leq \e^{-\frac{k_1+k_2}{2}t}\left(\e^{\frac{k_2-k_1}{2}t}-1\right)|\nabla f|(X_t).$$
By this and Eq.~\eqref{add-1}, we have
\begin{align}\label{add-eq-1}
  &\left|2(a\nabla f+b\nabla P_tf)-Q_t\partrinv0t\nabla f(X_t)\right|^2\notag\\
  &\quad\leq \l[\left|2(a\nabla f+b\nabla P_tf)-\e^{-\frac{k_1+k_2}{2}t}\partrinv0t\nabla f(X_t)\right|+\e^{-\frac{k_1+k_2}{2}t}\left(\e^{\frac{k_2-k_1}{2}t}-1\right)|\nabla f|(X_t)\r]^2\notag\\
  &\quad=\left|2(a\nabla f+b\nabla P_tf)-\e^{-\frac{k_1+k_2}{2}t}\partrinv0t\nabla f(X_t)\right|^2\notag\\
  &\qquad +
    2\e^{-\frac{k_1+k_2}{2}t}\left(\e^{\frac{k_2-k_1}{2}t}-1\right)\left|2(a\nabla f+b\nabla P_tf)-\e^{-\frac{k_1+k_2}{2}t}\partrinv0t\nabla f(X_t)\right|\,|\nabla f|(X_t)\notag\\
  &\qquad +\e^{-(k_1+k_2)t}\left(\e^{\frac{k_2-k_1}{2}t}-1\right)^2|\nabla f|^2(X_t).
\end{align}
By Cauchy's inequality, we have
\begin{align*}
  &2\e^{-\frac{k_1+k_2}{2}t}\left(\e^{\frac{k_2-k_1}{2}t}-1\right)\left|2(a\nabla f+b\nabla P_tf)-\e^{-\frac{k_1+k_2}{2}t}\partrinv0t\nabla f(X_t)\right|\,|\nabla f|(X_t)\\
  &\quad=2\sqrt{\e^{\frac{k_2-k_1}{2}t}-1}~\left|2(a\nabla f+b\nabla P_tf)-\e^{-\frac{k_1+k_2}{2}t}\partrinv0t\nabla f(X_t)\right| \e^{-\frac{k_1+k_2}{2}t}\sqrt{\e^{\frac{k_2-k_1}{2}t}-1}~|\nabla f|(X_t)\\
  &\quad\leq \left(\e^{\frac{k_2-k_1}{2}t}-1\right)\left|2(a\nabla f+b\nabla P_tf)-\e^{-\frac{k_1+k_2}{2}t}\partrinv0t\nabla f(X_t)\right|^2+
    \e^{-(k_1+k_2)t}\left(\e^{\frac{k_2-k_1}{2}t}-1\right)|\nabla f|^2(X_t).
\end{align*}
Thus, combining this inequality with \eqref{add-eq-1}, we obtain
\begin{align*}
  &\left|2(a\nabla f+b\nabla P_tf)-Q_t\partrinv0t\nabla f(X_t)\right|^2\\
  &\quad\leq\e^{\frac{k_2-k_1}{2}t}\left|2(a\nabla f+b\nabla P_tf)-\e^{-\frac{k_2+k_1}{2}t}\partrinv0t\nabla f(X_t)\right|^2+\e^{-(k_2+k_1)t}\left(\e^{\frac{k_2-k_1}{2}t}-1\right)\e^{\frac{k_2-k_1}{2}t}|\nabla f|^2(X_t)\\
  &\quad\leq 4\e^{\frac{k_2-k_1}{2}t}|a\nabla f+b\nabla P_tf|^2-4\e^{-k_1t}\big<a\nabla f+b\nabla P_tf, \partrinv0t\nabla f(X_t)\big>
    +\e^{-2k_1t}|\nabla f|^2(X_t).
\end{align*}
Expanding the terms above yields
\begin{align}
  &\left|Q_t\partrinv0t\nabla f(X_t)\right|^2-\e^{-2k_1t}|\nabla f|^2(X_t)\notag\\
  &\leq 4\l[\left(\e^{\frac{k_2-k_1}{2}t}-1\right)|a \nabla f+b\nabla P_tf|^2+\big<a\nabla f+b \nabla P_tf,Q_t\partrinv0t\nabla f(X_t)-\e^{-k_1t}\partrinv0t\nabla f(X_t)\big>\r].\label{Eq:Estimate}
\end{align}
We observe that
$|\nabla P_tf|^2\leq \E(|Q_t\partrinv0t\nabla f(X_t)|^2)$.  Hence, by
taking expectation on both sides of inequality~\eqref{Eq:Estimate}, we arrive at
\begin{align}\label{add-3}
  &|\nabla P_tf|^2-\e^{-2k_1t}P_t|\nabla f|^2\notag\\
  &\leq 4\l[\left(\e^{\frac{k_2-k_1}{2}t}-1\right)|a \nabla f+b\nabla P_tf|^2+\big<a\nabla f+b \nabla P_tf, \nabla P_tf-\e^{-k_1t}\E \partrinv0t\nabla f(X_t)\big>\r].
\end{align}
Thus, letting $a=1$, $b=0$, respectively $a=0$, $b=1$, we complete the
proof of (ii) and (ii').  \smallskip

{\bf(b) }\ $\text{(i)}\Rightarrow\text{(iii),\,(iii')}$: By It\^{o}'s
formula, we have
\begin{align}\label{eq2}
  \vd (P_{t-s}f^{2/p})^p(X_s)&=\vd M_s+(L+\partial_s)\left(P_{t-s}f^{2/p}(X_{s})\right)^p\,\vd s\notag\\
                             &=\vd M_s+p(p-1)\left(P_{t-s}f^{2/p}(X_{s})\right)^{p-2}\,|\nabla P_{t-s}f^{2/p}|^2(X_s)\,\vd s
\end{align}
where $M_s$ is a local martingale.  In addition,
\begin{align}\label{eq3}
  \left|\nabla P_{t-s}f^{2/p}(X_s)\right|^2&=\left|\partr0s\E\left[\partrinv0s Q_{s,t}\partrinv st\nabla f^{2/p}(X_t)|\mathscr{F}_s\right]\right|^2\notag\\
                                           &=\frac{4}{p^2}\left|\E\left[f^{(2-p)/p}(X_t)\partrinv 0s Q_{s,t}\partrinv st\nabla f(X_t)|\mathscr{F}_s\right]\right|^2\notag\\
                                           &\leq \frac{4}{p^2}(P_{t-s}f^{2(2-p)/p})(X_s)\,\E\left[|Q_{s,t}\partrinv st\nabla f(X_t)|^2|\mathscr{F}_s\right],
\end{align}
where for fixed $s\geq0$, the two-parameter family $Q_{s,t}$ of random
automorphisms of $T_{X_s}M$ solves the pathwise equation
$$\frac{\vd Q_{s,t}}{\vd t}=-Q_{s,t}\,\Ric^Z_{\partr st},\quad Q_{s,s}=\id_{{X_s}}, \quad  t\geq s.$$
Analogously to Eq.~\eqref{Eq:RicZ0t} we have
$\Ric^Z_{\partr st}=\partrinv st\circ\Ric^Z_{X_t}\circ\partr st$.

As $2-p\in [0,1]$, by Jensen's inequality, we first observe
$$P_{t-s}f^{2(2-p)/p}\leq (P_{t-s}f^{2/p})^{2-p}.$$
Combining this with \eqref{eq2} and \eqref{eq3}, we obtain
\begin{align*}
  \vd (P_{t-s}f^{2/p})^p\leq \vd M_s+\frac{4(p-1)}{p}\,\E\left[|Q_{s,t}\partrinv st\nabla f(X_t)|^2|\mathscr{F}_s\right]\vd s.
\end{align*}
Integrating both sides from $0$ to $t$ and taking expectation, we
arrive at
\begin{align}\label{eq4}
  \frac{p(P_tf^2-(P_tf^{2/p})^p)}{4(p-1)}\leq \int_0^t\E\left[|Q_{s,t}\partrinv st\nabla f(X_t)|^2|\mathscr{F}_s\right]\vd s.
\end{align}
Now, using similar arguments as in \textbf{(a)}, we obtain
\begin{align}\label{eq5}
  &\E\left[|Q_{s,t}\partrinv st\nabla f(X_t)|^2|\mathscr{F}_s\right]\notag\\
  &\quad\leq \e^{-2k_1(t-s)}P_{t-s}|\nabla f|^2(X_s)+ 4\l(\e^{\frac{k_2-k_1}{2}(t-s)}-1\r)|\nabla f|^2(X_s)\notag\\
  &\qquad +4\E\l[\big<\nabla f(X_s), \nabla P_{t-s}f(X_s)-\e^{-k_1(t-s)}\partrinv st\nabla f(X_t)\big>\big|\mathscr{F}_s\r]
\end{align}
and
\begin{align}\label{eq6}
  &\E\left[|Q_{s,t}\partrinv st\nabla f(X_t)|^2|\mathscr{F}_s\right]\notag\\
  &\quad\leq \e^{-2k_1(t-s)}P_{t-s}|\nabla f|^2(X_s)+ 4\e^{\frac{k_2-k_1}{2}(t-s)}|\nabla P_{t-s}f|^2(X_s)\notag\\
  &\qquad -4\e^{-k_1(t-s)}\E\l[\l<\nabla f(X_s), \partrinv st\nabla f(X_t)\r>\big|\mathscr{F}_s\r].
\end{align}
Together with \eqref{eq4}, the proof of (iii) and (iii') is
completed.\smallskip

{\bf (c) }\ $\text{(i)}\Rightarrow\text{(iv) and (iv')}$: By It\^{o}'s
formula, we have
\begin{align}\label{eq7}
  \vd (P_{t-s}f^2)(X_s)\log (P_{t-s}f^2)(X_s)&=\vd \tilde{M_s}+(L+\partial_s)(P_{t-s}f^2)(X_s)\log (P_{t-s}f^2)(X_s)\,\vd s\notag\\
                                             &=\vd \tilde{M_s}+\frac{1}{P_{t-s}f^2(X_s)}|\nabla P_{t-s}f^2|^2(X_s)\,\vd s
\end{align}
where $\tilde{M}_s$ is a local martingale. Furthermore, using the
derivative formula, we have
\begin{align*}
  |\nabla P_{t-s}f^2|^2(X_s)&=\left|\E\left[\partrinv0s Q_{s,t}\partrinv st\nabla f^2(X_t)|\mathscr{F}_s\right]\right|^2\\
                            &\leq 4P_{t-s}f^2(X_s)\E\left[|Q_{s,t}\partrinv st\nabla f(X_t)|^2|\mathscr{F}_s\right].
\end{align*}
Combining this with \eqref{eq7}, we obtain
\begin{align*}
  \vd (P_{t-s}f^2)(X_s)\log (P_{t-s}f^2)(X_s)\leq \vd \tilde{M_s}+4\E\left[|Q_{s,t}\partrinv st\nabla f(X_t)|^2|\mathscr{F}_s\right]\vd s.
\end{align*}
Using the estimates in \eqref{eq5} and \eqref{eq6} for
$\E[|Q_{s,t}\partrinv st\nabla f(X_t)|^2|\mathscr{F}_s]$, we finish
the proof by integrating from $0$ to $t$ and taking expectation on
both sides.
\end{proof}

\begin{remark}
  Actually, when $k_1\neq k_2$, the following inequality can be
  derived by minimizing the upper bound in \eqref{add-3} over $a,b$
  under the restriction $a+b=1$:
  \begin{align}\label{add-2}
    |\nabla P_tf|^2-\e^{-2k_1t}P_t|\nabla f|^2&\leq \Bigg\{4\l[\left(\e^{\frac{k_2-k_1}{2}t}-1\right)|\nabla f|^2+\l<\nabla f, \nabla P_tf\r>-\e^{-k_1t}\l<\nabla f, \E \partrinv0t\nabla f(X_t)\r>\r]\notag\\
                                              &\quad-\frac{\l<\nabla P_tf-\nabla f, 2\left(\e^{\frac{k_2-k_1}{2}t}-1\right)\nabla f+\nabla P_tf-\e^{-k_1t}\E \partrinv0t\nabla f(X_t)\r>^2}{\left(\e^{\frac{k_2-k_1}{2}t}-1\right)|\nabla P_tf-\nabla f|^2}\Bigg\}\wedge 0\notag\\
                                              &=\Bigg\{4\l[\e^{\frac{k_2-k_1}{2}t}|\nabla P_tf|^2-\e^{-k_1t}\l<\nabla P_tf,\E \partrinv0t\nabla f(X_t)\r>\r]\notag\\
                                              &\quad-\frac{\l<\nabla P_tf-\nabla f, \left(2\e^{\frac{k_2-k_1}{2}t}-1\right)\nabla P_tf -\e^{-k_1t}\E \partrinv0t\nabla f(X_t)\r>^2}{\left(\e^{\frac{k_2-k_1}{2}t}-1\right)|\nabla P_tf-\nabla f|^2}\Bigg\}\wedge 0.
  \end{align}
  It is easy to see that this bound is sharper than the ones
  given in Theorem \ref{th1} (ii) and (ii').
\end{remark}

\begin{proof}
  Inequality \eqref{add-2} can be checked as follows.  First recall
  estimate \eqref{add-3}:
  \begin{align*}
    &|\nabla P_tf|^2-\e^{-2k_1t}P_t|\nabla f|^2\\
    &\leq 4\l[\left(\e^{\frac{k_2-k_1}{2}t}-1\right)|a \nabla f+b\nabla P_tf|^2+\big<a\nabla f+b \nabla P_tf, \nabla P_tf-\e^{-k_1t}\E \partrinv0t\nabla f(X_t)\big>\r].
  \end{align*}
  Taking $b=1-a$ in the terms of the right-hand side, we get
  \begin{align}\label{eq11}
    &4\l[\left(\e^{\frac{k_2-k_1}{2}t}-1\right)|a \nabla f+b\nabla P_tf|^2+\big<a\nabla f+b \nabla P_tf, \nabla P_tf-\e^{-k_1t}\E \partrinv0t\nabla f(X_t)\big>\r]\notag\\
    &=4\left[\left(\e^{\frac{k_2-k_1}{2}t}-1\right)|\nabla f-\nabla P_tf|^2 a^2+\big<\nabla f-\nabla P_tf, (2\e^{\frac{k_2-k_1}{2}t}-1)\nabla P_tf-\e^{-k_1t}\E \partrinv0t\nabla f(X_t)\big>a\right. \notag\\
    &\quad +\left.\e^{\frac{k_2-k_1}{2}t}|\nabla P_tf|^2-\e^{-k_1t}\big<\nabla P_tf, \E \partrinv0t\nabla f(X_t)\big>\right].
  \end{align}
  For the value
  \begin{align}\label{add-4}
    a=a_0=-\frac{\l<\nabla f-\nabla P_tf, (2\e^{\frac{k_2-k_1}{2}t}-1)\nabla P_tf-\e^{-k_1t}\E \partrinv0t\nabla f(X_t)\r>}{2\left(\e^{\frac{k_2-k_1}{2}t}-1\right)|\nabla f-\nabla P_tf|^2},
  \end{align}
  the expression in \eqref{eq11} reaches its minimum as a function of
  $a$:
  \begin{align*}
    &4\l[\e^{\frac{k_2-k_1}{2}t}|\nabla P_tf|^2-\e^{-k_1t}\l<\nabla P_tf, \E \partrinv0t\nabla f(X_t)\r>\r]\\
    &\qquad-\frac{\l<\nabla f-\nabla P_tf, (2\e^{\frac{k_2-k_1}{2}t}-1)\nabla P_tf-\e^{-k_1t}\E \partrinv0t\nabla f(X_t)\r>^2}{\left(\e^{\frac{k_2-k_1}{2}t}-1\right)|\nabla f-\nabla P_tf|^2}.
  \end{align*}
  Similarly, substituting $a=1-b$ in the terms on the left-hand side of
  Eq.~\eqref{eq11}, we get
  \begin{align}\label{eq12}
    &4\l[\left(\e^{\frac{k_2-k_1}{2}t}-1\right)|a \nabla f+b\nabla P_tf|^2+\l<a\nabla f+b \nabla P_tf, \nabla P_tf-\e^{-k_1t}\E \partrinv0t\nabla f(X_t)\r>\r]\notag\\
    &=4\Big[\left(\e^{\frac{k_2-k_1}{2}t}-1\right)|\nabla f-\nabla P_tf|^2 b^2+\l<\nabla f-\nabla P_tf, 2\left(\e^{\frac{k_2-k_1}{2}t}-1\right)\nabla f+\nabla P_tf-\e^{-k_1t}\E \partrinv0t\nabla f(X_t)\r>b \notag\\
    &\quad +\left(\e^{\frac{k_2-k_1}{2}t}-1\right)|\nabla f|^2+\l<\nabla f, \nabla P_tf-\e^{-k_1t}\E \partrinv0t\nabla f(X_t)\r>\Big].
  \end{align}
  It is easy to see that for
$$b=1-a_0=-\frac{\l<\nabla f-\nabla P_tf, 2\left(\e^{\frac{k_2-k_1}{2}t}-1\right)\nabla f+\nabla P_tf-\e^{-k_1t}\E \partrinv0t\nabla f(X_t)\r>}{2\left(\e^{\frac{k_2-k_1}{2}t}-1\right)|\nabla f-\nabla P_tf|^2},$$
expression \eqref{eq12} reaches its minimal value:
\begin{align*}
  &4\l[\left(\e^{\frac{k_2-k_1}{2}t}-1\right)|\nabla f|^2+\l<\nabla f, \nabla P_tf\r>-\e^{-k_1t}\l<\nabla f, \E \partrinv0t\nabla f(X_t)\r>\r]\\
  &\qquad -\frac{\l<\nabla P_tf-\nabla f, 2\left(\e^{\frac{k_2-k_1}{2}t}-1\right)\nabla f+\nabla P_tf-\e^{-k_1t}\E \partrinv0t\nabla f(X_t)\r>^2}{\left(\e^{\frac{k_2-k_1}{2}t}-1\right)|\nabla P_tf-\nabla f|^2}.
\end{align*}
As the minimum is unique, we conclude that the upper bounds
\eqref{eq11} and \eqref{eq12} are indeed equivalent.
\end{proof}

To prove that the inequalities in (ii)-(iv), (ii')-(iv') imply
condition (i), we use the following lemma.

\begin{lemma}\label{lem1}
  For $x\in M$, let $X\in T_x M$ with $|X|=1$. Let
  $f\in C_0^{\infty}(M)$ such that $\nabla f(x)=X$ and $\Hess_f(x)=0$,
  and let $f_n=n+f$ for $n\geq 1$. Then,
  \begin{enumerate}[\rm(i)]
  \item for any $p>0$,
    \begin{align*}
      \Ric^Z(X,X)
      &=\lim_{t\rightarrow 0}\frac{P_{t}|\nabla f|^p(x)-|\nabla P_{t}f|^p(x)}{pt};
    \end{align*}
  \item for any $p>1$,
    \begin{align*}
      \Ric^Z(X,X)
      &=\lim_{n\rightarrow\infty}\lim_{t\rightarrow 0}\frac{1}{t}\l(P_{t}|\nabla f_n|^2-\frac{p\left\{P_{t}f_n^2-(P_{t}f_n^{{2}/{p}})^p\right\}}{4(p-1)t}\r)(x);
    \end{align*}
  \item $\Ric^Z(X,X)$ can be calculated as
    \begin{align*}
      \Ric^Z(X,X)=\lim_{n\rightarrow\infty}\lim_{t\rightarrow 0}\frac{1}{4t^2}\l\{4tP_{t}|\nabla f_n|^2+(P_{t}f_n^2)\log
      P_{t}f_n^2-P_{t}{f_n^2\log f_n^2}\r\}(x);
    \end{align*}
  \item $\Ric^Z(X,X)$ is also given by the following two limits:
    \begin{align*}
      \Ric^Z(X, X)=&\lim_{t\rightarrow0}\frac{\l\{\big<\nabla f, \E \partrinv0t\nabla f(X_t)\big>-\l<\nabla f, \nabla P_t f\r>\r\}(x)}{t}\\
      =&\lim_{t\rightarrow 0}\frac{\l\{\big<\nabla P_tf, \E \partrinv0t\nabla f(X_t)\big>-|\nabla P_tf|^2\r\}(x)}{t}.
    \end{align*}
  \end{enumerate}
\end{lemma}

\begin{proof}
  The formulae in (i)--(iii) can be found in \cite[Theorem
  2.2.4]{Wbook2} (see also \cite{Bakry, Sturm}).  The two expressions
  in (iv) are easily derived using Taylor expansions:
  \begin{align*}
    &\big<\nabla f, \E \partrinv0t\nabla f(X_t)\big>(x)-\l<\nabla f, \nabla P_tf\r>(x)\\
    &\quad=(\l<\nabla f, L\nabla f\r>(x)-\l<\nabla f, \nabla Lf\r>(x))t+\text{\rm o}(t)\\
    &\quad=\Ric^Z(\nabla f, \nabla f)(x)\,t+\text{\rm o}(t)
  \end{align*}
  and
  \begin{align*}
    &\big<\nabla P_tf, \E \partrinv0t\nabla f(X_t)\big>(x)-\l<\nabla P_tf, \nabla P_tf\r>(x)\\
    &\quad=(\l<\nabla f, L\nabla f\r>(x)-\l<\nabla f, \nabla Lf\r>(x))t+\text{\rm o}(t)\\
    &\quad=\Ric^Z(\nabla f, \nabla f)(x)\,t+\text{\rm o}(t).
  \end{align*}
  Here, we use the fact that for $f\in C_0^{\infty}(M)$ such that
  $\Hess_f(x)=0$, the following equation holds:
  \begin{equation*}\Ric^Z(\nabla f, \nabla f)(x)=\l<L\nabla f, \nabla
    f\r>(x)-\l<\nabla Lf, \nabla f\r>(x).\qedhere
  \end{equation*}
\end{proof}

Using Lemma \ref{lem1}, we are now able to complete the proof of the
main result.
\begin{proof}[Proof of Theorem \ref{th1}.]

  {\bf Part II}\ \
  ``(ii) and (ii') $\Rightarrow$ (i)'':\\
  Fix $x\in M$ and let $f\in C_0^{\infty}(M)$ such that
  $\Hess_f(x)=0$.  Without explicit mention, the following computations are
  all taken implicitly at the point $x$.  First, we rewrite the
  inequalities (ii) and (ii') as follows,
  \begin{align*}
    &\frac{|\nabla P_tf|^2-P_t|\nabla f|^2}{2t}+\frac{1-\e^{-2k_1t}}{2t}P_t|\nabla f|^2\\
    & \leq \frac2t\left(\e^{\frac{k_2-k_1}{2}t}-1\right)|a\nabla f+b\nabla P_tf|^2+2\frac{\l<a\nabla f+b\nabla P_tf, \nabla P_tf\r>-\big<a\nabla f+b\nabla P_tf, \E \partrinv0t\nabla f(X_t)\big>}{t}\\
    &\quad +\frac2t \left(1-\e^{-k_1t}\right)\E\big<a\nabla f+b\nabla P_tf, \partrinv0t\nabla f(X_t)\big>
  \end{align*}
  where $a=1$, $b=0$ or $a=0$, $b=1$.  Letting $t\rightarrow 0$, by
  Lemma \ref{lem1}, we obtain
$$-\Ric^Z(\nabla f, \nabla f)+k_1|\nabla f|^2\leq (k_2-k_1)|\nabla f|^2-2\Ric^Z(\nabla f, \nabla f)+2k_1|\nabla f|^2$$
which implies that
$$\Ric^Z(\nabla f, \nabla f)\leq k_2|\nabla f|^2.$$

``(iii), (iv), (iii'), (iv') $\Rightarrow $ (i)'': We only prove that
``(iii) and (iii') imply (i)'', as the inequalities (iv) and (iv') can
be considered as limits of the inequalities (iii) and (iii') as
$p\downarrow 1$.

For $x\in M$ and $f\in C^{\infty}_0(M)$ such that $\Hess_f(x)=0$, let
$f_n:=f+n$ and rewrite (iii) as
\begin{align}\label{eq8}
  &\frac{1}{t^2}\l(\frac{p(P_tf_n^2-(P_tf_n^{2/p})^p)}{4(p-1)}-tP_t|\nabla f_n|^2\r)
    -\frac1{t^2}{\int_0^t[1-\e^{-2k_1(t-s)}]\vd s}\times P_t|\nabla f_n|^2\notag\\
  &\leq \frac{4}{t^2}\int_0^t\l(\e^{\frac{k_2-k_1}{2}(t-r)}-1\r)P_r|\nabla f_n|^2\,\vd r
    +\frac{4}{t^2}\int_0^t\l(1-\e^{-k_1(t-r)}\r)\E\l<\nabla f_n(X_r), \partrinv rt\nabla f_n(X_t)\r>\,\vd r\notag\\
  &\quad +\frac{4}{t^2}\int_0^t\E\l<\nabla f_n(X_r), \nabla P_{t-r}f_n(X_r)- \partrinv rt\nabla f_n(X_t)\r>\,\vd r.
\end{align}
Now letting $t\rightarrow 0$, by Lemma \ref{lem1} (ii), the terms on
the right-hand side become
$$-\Ric^Z(\nabla f, \nabla f)+k_1|\nabla f|^2.$$
For the terms on the left-hand side of \eqref{eq8}, we have the
following expansions:
\begin{align*}
  \frac{4}{t^2}\int_0^t\l(\e^{\frac{k_2-k_1}{2}(t-r)}-1\r)P_r|\nabla f_n|^2\,\vd r
  &=
    \frac{4}{t^2}\int_0^t\l(\e^{\frac{k_2-k_1}{2}(t-r)}-1\r)(|\nabla f_n|^2+\text{\rm o}(1))\,\vd r\\                     &=(k_2-k_1)|\nabla f|^2+\text{\rm o}(1);\\
  \frac{4}{t^2}\int_0^t\l(1-\e^{-k_1(t-r)}\r)\E\l<\nabla f_n(X_r), \partrinv rt\nabla f_n(X_t)\r>\,\vd r
  &=\frac{4}{t^2}\int_0^t\l(1-\e^{-k_1(t-r)}\r)(|\nabla f_n|^2+\text{\rm o}(1))\,\vd r\\
  &=2k_1|\nabla f|^2+\text{\rm o}(1);
\end{align*}
\begin{align*}
  \frac{4}{t^2}\int_0^t\E\l<\nabla f_n(X_r), \nabla P_{t-r}f_n(X_r)- \partrinv rt\nabla f_n(X_t)\r>\,\vd r
  &=\frac{4}{t^2}\int_0^t(\Ric^Z(\nabla f_n, \nabla f_n)(t-r)+\text{\rm o}(t)+\text{\rm o}(r))\,\vd r\\
  &=2\Ric^Z(\nabla f, \nabla f)+\text{\rm o}(1).
\end{align*}
Therefore, letting $t\rightarrow 0$ in \eqref{eq8}, we arrive at
$$-\Ric^Z(\nabla f, \nabla f)+k_1|\nabla f|^2\leq (-2\Ric^Z(\nabla f, \nabla f)+(k_2+k_1)|\nabla f|^2)\wedge 0,$$
i.e.,
$$k_1|\nabla f|^2\leq \Ric^Z(\nabla f, \nabla f)\leq k_2|\nabla f|^2.$$
The proof of ``(iii') implies (i)'' is similar. We skip the details
here.
\end{proof}

\begin{remark}\label{rem2}
  In the proof of Theorem \ref{th1}
  ``$\text{(ii) (ii')}\Rightarrow\text{(i)}$", we take into account
  that for $a$ and $b$ satisfying $a+b=1$, trivially
  $\lim_{t\rightarrow 0}(a\nabla f+b\nabla P_tf)=\nabla f$
  holds. However, when choosing $a=a_0$ as in \eqref{add-4} for the
  proof of inequality \eqref{add-3}, obviously $a_0$ depends on $t$,
  and thus we get
  \begin{align*}
    \lim_{t\rightarrow 0} &\left(a_0\nabla f+(1-a_0)\nabla P_t f\right)\\
                          &=\lim_{t\rightarrow 0} (\nabla f+ (1-a_0)(\nabla P_tf-\nabla f))\\
                          &=\nabla f-\lim_{t\rightarrow0}\frac{\l<\nabla f-\nabla P_tf, 2\left(\e^{\frac{k_2-k_1}{2}t}-1\right)\nabla f+\nabla P_tf-\e^{-k_1t}\E \partrinv0t\nabla f(X_t)\r>}{2\left(\e^{\frac{k_2-k_1}{2}t}-1\right)|\nabla f-\nabla P_tf|^2}\,(\nabla P_tf-\nabla f)\\
                          &=\nabla f+\lim_{t\rightarrow 0}\frac{\l<(\nabla Lf)t+\text{\rm o}(t), k_2\nabla f t+(\nabla Lf)t-(L\nabla f) t +\text{\rm o}(t)\r>}{(k_2-k_1)|\nabla Lf|^2t^3+\text{\rm o}(t^3)}(\nabla Lf) t\\
                          &=\nabla f+\frac{\l<\nabla Lf, k_2\nabla f+\nabla Lf-L\nabla f\r>}{(k_2-k_1)|\nabla Lf|^2}\nabla Lf\neq \nabla f.
  \end{align*}
  Actually, dividing both hands of inequality \eqref{add-2} by $2t$
  and letting $t\rightarrow 0$, we obtain
$$k_1|\nabla f|^2\leq \Ric(\nabla f,\nabla f)\leq k_2|\nabla f|^2-\frac{\l<\nabla Lf, k_2\nabla f+\nabla Lf-L\nabla f\r>^2}{(k_2-k_1)|\nabla Lf|^2}\ ({}\leq k_2|\nabla f|^2).$$
\end{remark}

\section{Pointwise characterizations of curvature
  bounds}\label{Sect:3}

Consider a Riemannian manifold $M$ possibly with non-empty boundary
$\partial M$, and let $X_t$ be a reflecting diffusion processes
generated by $L=\Delta+Z$. We assume that $X_t$ is non-explosive. It
is well known that the reflecting process $X_t$ can be constructed as
solution to the equation
$$\vd X_t=\sqrt{2}u_t\circ\vd B_t+Z(X_t)\vd t+N(X_t)\vd l_t,$$
where $u_t$ is a horizontal lift of $X_t$ to the orthonormal frame
bundle, $N$ the inward normal unit vector field on $\partial M$ and
$l_t$ the local time of $X_t$ supported on $\partial M$, see
\cite{Wbook2} for details. Again,
$$\partr rs=u_s\circ u_r^{-1}\colon T_{X_r}M\to T_{X_s}M,\quad r\leq s,$$
denotes parallel transport along $t\mapsto X_t$.  Finally, let $\II$
be the second fundamental form of the boundary:
\begin{align*}
  \II(X,Y)=-\l<\nabla_X N, Y\r>,\quad \text{$X,Y\in T_x\partial M,\ x\in\partial M$}.
\end{align*}

In this section, we extend the results of Section \ref{Sect:2} in
order to characterize pointwise bounds on $\Ric^Z$ and $\II$.  To this
end, for continuous functions $K_1, K_2, \sigma_1$ and $\sigma_2$ on
$M$, let
$$\varK_1(X_{[s,t]})=\int_s^tK_1(X_r)\,\vd r+\sigma_1(X_r)\,\vd l_r,\quad
\varK_2(X_{[s,t]})=\int_s^tK_2(X_r)\,\vd r+\sigma_2(X_r)\,\vd l_r$$
where $X_{[s,t]}=\{X_r:r\in [s,t]\}$.  Furthermore, let
$$C_N^{\infty}(M):=\{f\in C_0^{\infty}(M)\colon\,Nf|_{\partial M}=0\}.$$
Finally let
$$(P_tf)(x)=\E[f(X_t^x)],\quad f\in C_b(M),$$
be the semigroup with Neumann boundary conditions generated by $L$.

The result of this section can be presented as follows.

\begin{theorem}\label{th2}
  We keep the assumptions and notations from above.  Let
  $x\mapsto K_1(x)$ and $x\mapsto K_2(x)$ be two continuous functions
  on $M$ such that $K_1\leq K_2$. In addition, let
  $x\mapsto \sigma_1(x)$ and $x\mapsto \sigma_2(x)$ be two functions
  on $\partial M$ such that $\sigma_1\leq \sigma_2$.  Assume that
  \begin{align}\label{Icd}
    \E\big[\e^{-(2+\varepsilon)\varK_1(X_{[0,t]})}\big]<\infty, \quad
    \text{for some }\varepsilon>0\text{ and }t>0.
  \end{align}
  The following statements are equivalent:
  \begin{enumerate}[(i)]
  \item [\rm(i)] Curvature $\Ric^Z$ and second fundamental form $\II$
    satisfy the bounds
 $$K_1(x)\leq \Ric^Z(x)\leq K_2(x), \quad x\in M,\quad \mbox{and} \quad \sigma_1(x) \leq \II(x)\leq \sigma_2(x), \quad x\in \partial M.$$
\item [\rm(ii)] For $f\in C_N^{\infty}(M)$ and $t>0$,
  \begin{align*}
    \quad&|\nabla P_tf|^2-\E\left[\e^{-2\varK_1(X_{[0,t]})}|\nabla f|^2(X_t)\right]\\
         & \leq 4\l\{\l(\E\e^{\frac{1}{2}(\varK_2(X_{[0,t]})-\varK_1(X_{[0,t]}))}-1\r)|\nabla f|^2+\l<\nabla f, \nabla P_tf\r>-\l<\nabla f,\E\big[\e^{-\varK_1(X_{[0,t]})} \partrinv0t\nabla f(X_t)\big]\r>\r\}\wedge 0.
  \end{align*}
\item [\rm(ii')] For $f\in C_N^{\infty}(M)$ and $t>0$,
  \begin{align*}
    &|\nabla P_tf|^2-\E\e^{-2\varK_1(X_{[0,t]})}|\nabla f|^2(X_t)\\
    &\quad\leq 4\l\{\E\e^{\frac{1}{2}(\varK_2(X_{[0,t]})-\varK_1(X_{[0,t]}))}|\nabla P_tf|^2-\l<\nabla P_tf, \E\big[\e^{-\varK_1(X_{[0,t]})}\partrinv0t\nabla f(X_t)\big]\r>\r\}\wedge 0.
  \end{align*}
\item [\rm(iii)] For $f\in C_N^{\infty}(M)$, $p\in (1,2]$ and $t>0$,
  \begin{align*}
    & \frac{p(P_tf^2-(P_tf^{2/p})^p)}{4(p-1)}
      -\E\l[\int_0^t\e^{-2\varK_1(X_{[r,t]})}\,\vd r\times |\nabla f|^2(X_t)\r]\\
    &\quad\leq 4\int_0^t\l(\E\e^{\frac{1}{2}(\varK_2(X_{[r,t]})-\varK_1(X_{[r,t]}))}-1\r)P_r|\nabla f|^2\\
    &\qquad +\E\l<\nabla f(X_r),\nabla P_{t-r}f(X_r)-\e^{-\varK_1(X_{[r,t]})}\partrinv rt\nabla f(X_t)\r>\,\vd r\wedge 0.
  \end{align*}
\item [\rm(iii')] For $f\in C_N^{\infty}(M)$, $p\in (1,2]$ and $t>0$,
  \begin{align*}
    & \frac{p(P_tf^2-(P_tf^{2/p})^p)}{4(p-1)}-\E\l[\int_0^t\e^{-2\varK_1(X_{[r,t]})}\,\vd r\times |\nabla f|^2(X_t) \r]\\
    &\quad\leq 4\int_0^t\E\l[\e^{\frac{1}{2}(\varK_2(X_{[r,t]})-\varK_1(X_{[r,t]}))}\r]P_r|\nabla P_{t-r}f|^2-\E\l[\e^{-\varK_1(X_{[r,t]})}\l<\nabla f(X_r),\partrinv rt\nabla f(X_t)\r>\r]\,\vd r\wedge 0.
  \end{align*}
\item [\rm(iv)] For $f\in C_N^{\infty}(M)$ and $t>0$,
  \begin{align*}
    & \frac{1}{4}\left(P_t(f^2\log f^2)-P_tf^2\log P_tf^2\right)-\E\l[\int_0^t\e^{-2\varK_1(X_{[r,t]})}\,\vd r\times |\nabla f|^2(X_t) \r]\\
    &\quad\leq 4\int_0^t\l(\E\e^{\frac{1}{2}(\varK_2(X_{[r,t]})-\varK_1(X_{[r,t]}))}-1\r)P_r|\nabla f|^2\\
    &\qquad +\E\l<\nabla f(X_r),\nabla P_{t-r}f(X_r)-\e^{-\varK_1(X_{[r,t]})}\partrinv rt\nabla f(X_t)\r>\,\vd r\wedge 0.
  \end{align*}
\item [\rm(iv')] For $f\in C_N^{\infty}(M)$ and $t>0$,
  \begin{align*}
    & \frac14\left(P_t(f^2\log f^2)-P_tf^2\log P_tf^2\right)-\E\l[\int_0^t\e^{-2\varK_1(X_{[r,t]})}\,\vd r\times|\nabla f|^2(X_t) \r]\\
    &\quad\leq 4\int_0^t\E\l[\e^{\frac{1}{2}(\varK_2(X_{[r,t]})-\varK_1(X_{[r,t]}))}\r]\,P_r|\nabla P_{t-r}f|^2-\E\l[\e^{-\varK_1(X_{[r,t]})}\l<\nabla f(X_r),\partrinv rt\nabla f(X_t)\r>\r]\,\vd r\wedge 0.
  \end{align*}
\end{enumerate}
\end{theorem}

To prove the theorem, we need the following lemmas.
\begin{lemma}\label{lem5}
  ({\cite[Lemma 3.1.2]{Wbook2}}) Let $X_t^x$ be the reflecting
  diffusion process generated by $L$ such that $X_0=x$ and $l_t^x$ the
  corresponding local time on the boundary.
  \begin{enumerate}[\rm(i)]
  \item For any $x\in M$ and $r_0>0$, there exists a constant $c>0$
    such that
$$\P\{\sigma_r\leq t\}\leq \e^{-cr^2/t}, \quad \text{for all }r\in [0,r_0]\text{ and } t>0,$$
where $\sigma_r=\inf\{s\geq 0\colon\, \rho(x,X^x_s)\geq r\}$.
\item Let $x\in \partial M$ and $r$ as above. Then:
  \begin{enumerate}[\rm(a)]
  \item $\E^x[\e^{\lambda l_{t\wedge \sigma_r}}]<\infty$ for any
    $\lambda >0$ and there exists $c>0$ such that
    $\E^x[l_{t\wedge \sigma _r}^2]\leq c(t+t^2)$;
  \item
    $\E^x[l_{t\wedge \sigma_r}]=\frac{2\sqrt{t}}{\sqrt{\pi}}+\text{\rm
      o}(t^{1/2})$ holds for small $t>0$.
  \end{enumerate}
\end{enumerate}
\end{lemma}

By means of Lemma \ref{lem5}, we can derive pointwise formulae for
$\Ric^Z$ and $\II$.
\begin{lemma}\label{lemma6}
  Let $x\in \mathring{M}=: M\setminus \partial M$ and $X\in T_xM$ with
  $|X|=1$.  Let $f\in C_0^{\infty}(M)$ such that $Nf|_{\partial M}=0$,
  $\Hess_f(x)=0$ and $\nabla f(x)=X$ and let $f_n=f+n$ for $n\geq
  1$. Then all assertions of Lemma \ref{lem1} hold.
\end{lemma}

\begin{proof}
  Let $r>0$ be such that $B(x,r)\subset \mathring{M}$ and
  $|\nabla f|\geq \frac{1}{2}$ on $B(x,r)$.  Due to Lemma \ref{lem5},
  the proof of Lemma \ref{lem1} applies to the present situation,
  using $t\wedge \sigma_r$ to replace $t$, so that the boundary
  condition is avoided.  We refer the reader to the proof of
  \cite[Theorem 3.2.3]{Wbook2} for more explanation.
\end{proof}

\begin{lemma}\label{II-form}
  Let $x\in \partial M$ and $X\in T_xM$ with $|X|=1$.
  \begin{enumerate}[\rm(1)]
  \item For any $f\in C_0^{\infty}(M)$ such that $\nabla f(x)=X$, and
    for any $p>0$, we have
    \begin{align}
      \II(X,X)&=\lim_{t\downarrow 0}\frac{\sqrt{\pi}}{2p\sqrt{t}}\l\{P_{t}|\nabla f|^p-|\nabla f|^p\r\}(x)\nonumber\\
              &=\lim_{t\downarrow
                0}\frac{\sqrt{\pi}}{2p\sqrt{t}}\l\{P_{t}|\nabla f|^p-|\nabla P_{t}f|^p\r\}(x)\notag\\
              &=\lim_{t\rightarrow0}\frac{\sqrt
                {\pi}\l\{\l<\nabla f, \E \partrinv0t\nabla f(X_t)\r>-\l<\nabla f, \nabla P_t f\r>\r\}(x)}{2\sqrt{t}}\label{add-eq-10}\\
              &=\lim_{t\rightarrow 0}\frac{\sqrt{\pi}\l\{\l<\nabla P_tf, \E \partrinv0t\nabla f(X_t)\r>-|\nabla P_tf|^2\r\}(x)}{2\sqrt{t}}.\label{add-eq-11}
    \end{align}
  \item If moreover $f>0$, then for any $p\in [1,2]$,
    \begin{align*}
      \II(X,X)&=-\lim_{t\downarrow
                0}\frac{3}{8}\sqrt{\frac{\pi}{t}}\l\{|\nabla f|^2+\frac{p[(P_{t}f^{2/p})^p-P_{t}f^2]}{4(p-1)t}\r\}(x)\\
              &=-\lim_{t\downarrow
                0}\frac{3}{8}\sqrt{\frac{\pi}{t}}\l\{|\nabla P_{t}f|^2+\frac{p[(P_{t}f^{2/p})^p-P_{t}f^2]}{4(p-1)t}\r\}(x),
    \end{align*}
    where when $p=1$, we interpret the quotient
    $\displaystyle\frac{(P_{t}f^{2/p})^p-P_{t}f^2}{p-1}$ as the limit
    \begin{align*}
      \lim_{p\downarrow
      1}\frac{(P_{t}f^{2/p})^p-P_{t}f^2}{p-1}=(P_{t}f^2)\log P_{t}f^2- P_{t}(f^2\log f^2).
    \end{align*}
  \end{enumerate}
\end{lemma}

\begin{proof}
  We only need to prove formulas \eqref{add-eq-10} and
  \eqref{add-eq-11}.  For the remaining statements we refer to
  \cite[Theorem 3.2.4]{Wbook2}.  Let $r>0$ such that
  $|\nabla f|\geq 1/2$ on $B(x,r)$, and let
  $\sigma_r:=\inf\{s\geq0: X_s\notin B(x,r)\}$. Then, by It\^o's
  formula and Lemma \ref{lem5}, we get
  \begin{align*}
    \E\big[\partrinv0t\nabla f(X_t)\big]
    &= \nabla f(x)+\E\l[\int_0^{t\wedge\sigma_r}\partrinv0s (\square+\nabla_Z) (\nabla f)(X_s)\,\vd s+ \partrinv0s\nabla_N(\nabla f)(X_s)\,\vd l_s\r]+\text{\rm o}(t)
  \end{align*}
  where $\square=-\nabla^*\nabla$ is the connection Laplacian (or
  rough Laplacian) acting on $\Gamma(TM)$.

  Along with Lemma \ref{lem5} (ii) (b), the formulae in
  \eqref{add-eq-10} and \eqref{add-eq-11} are obtained by taking into
  account the expansions:
 $$\big<\E\big[\partrinv0t\nabla f(X_t)\big], \nabla f\big>=|\nabla f|^2+\II(\nabla f, \nabla f)\frac{2\sqrt{t}}{\sqrt{\pi}}+\text{\rm o}(\sqrt{t}),$$
 resp.
 \begin{equation*}
   \big<\E\big[\partrinv0t\nabla f(X_t)\big], \nabla P_tf\big>=|\nabla f|^2+\II(\nabla f, \nabla f)\frac{2\sqrt{t}}{\sqrt{\pi}}+\text{\rm o}(\sqrt{t}).
   \qedhere
 \end{equation*}
\end{proof}

\begin{proof}[Proof of Theorem \ref{th2}.]
  Let $\Ric^Z(x)\geq K_1(x)$ and $\II(x)\geq
  \sigma_1(x)$. Furthermore, assume that
  \begin{align*}
    \E\left[\e^{-(2+\varepsilon)\varK_1(X_{[0,t]})}\right]<\infty,\quad\text{for some } \varepsilon>0 \text{ and } t>0.
  \end{align*}
  By \cite[Theorem 4.1.1]{Wbook2}, there exists a unique two-parameter
  family of random endomorphisms $Q_{s,t}\in\End(T_{X_s}M)$ solving,
  for $s\geq0$ fixed, the following equation in $t\geq s$,
$$\vd Q_{s,t}=-Q_{s,t}\left(\Ric_{\partr st}^Z\,\vd t+\II_{\partr st}\,\vd l_t\right)
(\id-\1_{\{X_t\in\partial M\}}P_{\partr st}),\quad Q_{s,s}=\id,$$
where by definition, for
$u\in \partial\OM:= \{u\in\OM: {\bf p}u\in \partial M\}$,
$$P(uy, uz)=\l<uy, N\r>\l<uz, N\r>,\quad y, z\in \R^d.$$
Recall that
$$\Ric^Z_{\partr st}=\partrinv st\circ\Ric^Z_{X_t}\circ\partr st,\quad
\II_{\partr st}=\partrinv st\circ\II_{X_t}\circ\partr st,\quad
P_{\partr st}=\partrinv st\circ P_{X_t}\circ\partr st,
$$
where as usual bilinear forms on $TM$, resp. on $T\partial M$, are
understood fiberwise as linear endomorphisms via the metric.
Moreover, by \cite[Theorem 3.2.1]{Wbook2}, we have
\begin{align}\label{Eq:GradEst}
  \nabla P_{t-s}f(X_s)=\partr 0s\E[\partrinv 0s Q_{s,t}\partrinv st\nabla f(X_t)|\mathscr{F}_s].
\end{align}
By using derivative formula~\eqref{Eq:GradEst}, the proofs are similar
to that of Theorem \ref{th1}.  We only prove the equivalence
``$\text{(i)}\Leftrightarrow\text{(ii)}$ or (iii)'' to explain the
idea.\smallskip

``$\text{(i)}\Rightarrow\text{(ii)}$'':\ \ First, from the derivative
formula and the lower bound on the curvature, we get
\begin{align}\label{eq2-1}
  |\nabla P_tf|^2\leq \E\left[\e^{-2\varK_1(X_{[0,t]})}|\nabla f|^2(X_t)\right].
\end{align}
Next, it is easy to see that
\begin{align}\label{eq1-0}
  2\nabla f&-Q_t\partrinv0t\nabla f(X_t)\notag\\
           &=2\nabla f-\e^{-\frac{1}{2}\left(\varK_2(X_{[0,t]})+\varK_1(X_{[0,t]})\right)}\partrinv0t\nabla f(X_t)\notag\\
           &\quad +\left(\e^{-\frac{1}{2}\left(\varK_2(X_{[0,t]})+\varK_1(X_{[0,t]})\right)}\id-Q_t\right)\partrinv0t\nabla f(X_t)
\end{align}
where $Q_t:=Q_{0,t}$, which implies that
\begin{align*}
  &\left|2\nabla f-Q_t\partrinv0t\nabla f(X_t)\right|\\
  &\qquad\leq \left|2\nabla f-\e^{-\frac{1}{2}\left(\varK_2(X_{[0,t]})+\varK_1(X_{[0,t]})\right)}\partrinv0t\nabla f(X_t)\right|\\
  &\qquad\quad+\left|(\e^{-\frac{1}{2}\left(\varK_2(X_{[0,t]})+\varK_1(X_{[0,t]})\right)}\id-Q_t)\partrinv0t\nabla f(X_t)\right|.
\end{align*}
We start by estimating the last term on the right-hand side,
\begin{align*}
  &\left|\left(\e^{-\frac{1}{2}\left(\varK_2(X_{[0,t]})+\varK_1(X_{[0,t]})\right)}\id-Q_t\right)\partrinv0t\nabla f(X_t)\right|\\
  &\quad\leq \e^{-\frac{1}{2}\left(\varK_2(X_{[0,t]})+\varK_1(X_{[0,t]})\right)}\left\|\id-\e^{\frac{1}{2}\left(\varK_2(X_{[0,t]})+\varK_1(X_{[0,t]})\right)}Q_t\right\|  |\nabla f(X_t)|.
\end{align*}
Observe that we may rewrite 
\begin{align*}
  \id&-\e^{\frac{1}{2}\left(\varK_2(X_{[0,t]})+\varK_1(X_{[0,t]})\right)}Q_t\\
     &=-\int_0^t\frac{\vd \big[\e^{\frac{1}{2}\left(\varK_2(X_{[0,s]})+\varK_1(X_{[0,s]})\right)}Q_{s}\big]}{\vd s}\vd s\\
     &=\int_0^t \e^{\frac{1}{2}\left(\varK_2(X_{[0,s]})+\varK_1(X_{[0,s]})\right)}Q_{s}\Bigg[\l(\Ric_{\partr0s}^Z-\frac{K_1(X_s)+K_2(X_s)}{2}\id\r)\left(\id-\1_{\{X_s\in\partial M\}}P_{\partr0s}\right)\,\vd s\\
     &\qquad +\l(\II_{\partr0s}-\frac{\sigma_1(X_s)+\sigma_2(X_s)}{2}\id\r)\left(\id-\1_{\{X_s\in\partial M\}}P_{\partr0s}\right)\,\vd l_s\Bigg].
\end{align*}
Thus we get
\begin{align*}
  &\left\|\id-\e^{\frac{1}{2}\left(\varK_2(X_{[0,t]})+\varK_1(X_{[0,t]})\right)}Q_t\right\|\\
  &\quad\leq \int_0^t \e^{\frac{1}{2}\left(\varK_2(X_{[0,s]})+\varK_1(X_{[0,s]})\right)}\|Q_{s}\|\bigg(\l|\Ric_{\partr0s}^Z-\frac{K_1(X_s)+K_2(X_s)}{2}\id\r|\,\vd s\\
  &\qquad +\l|\II_{\partr0s}-\frac{\sigma_1(X_s)+\sigma_2(X_s)}{2}\id\r|\,\vd l_s\bigg)\\
  &\quad\leq \int_0^t \e^{\frac{1}{2}\left(\varK_2(X_{[0,s]})-\varK_1(X_{[0,s]})\right)} \l(\frac{K_2(X_s)-K_1(X_s)}{2}\,\vd s+\frac{\sigma_2(X_s)-\sigma_1(X_s)}{2}\,\vd l_s\r)\\
  &\quad= \e^{\frac{1}{2}\left(\varK_2(X_{[0,t]})-\varK_1(X_{[0,t]})\right)}-1,
\end{align*}
which implies
\begin{align*}
  &\l|2\nabla f-Q_t\partrinv0t\nabla f(X_t)\r|^2\\
  & \quad\leq \bigg[\l|2\nabla f-\e^{-\frac{1}{2}\left(\varK_2(X_{[0,t]})+\varK_1(X_{[0,t]})\right)}\partrinv0t\nabla f(X_t)\r|\\
  &\qquad +\e^{-\frac{1}{2}\left(\varK_2(X_{[0,t]})+\varK_1(X_{[0,t]})\right)}\left(\e^{\frac{1}{2}(\varK_2(X_{[0,t]})-\varK_1(X_{[0,t]}))}-1\right)|\nabla f|(X_t)\bigg]^2\\
  &\quad\leq \e^{\frac{1}{2}\left(\varK_2(X_{[0,t]})-\varK_1(X_{[0,t]})\right)}\l|2\nabla f-\e^{-\frac{1}{2}(\varK_2(X_{[0,t]})+\varK_1(X_{[0,t]}))}\partrinv0t\nabla f(X_t)\r|^2\\
  &\qquad +\e^{-\left(\varK_2(X_{[0,t]})+\varK_1(X_{[0,t]})\right)}\left(\e^{\frac{1}{2}\left(\varK_2(X_{[0,t]})-\varK_1(X_{[0,t]})\right)}-1\right)\e^{\frac{1}{2}\left(\varK_2(X_{[0,t]})-\varK_1(X_{[0,t]})\right)}|\nabla f|^2(X_t)\\
  & \quad= 4\e^{\frac{1}{2}\left(\varK_2(X_{[0,t]})-\varK_1(X_{[0,t]})\right)}|\nabla f|^2-4\e^{-\varK_1(X_{[0,t]})}\l<\nabla f, \partrinv0t\nabla f(X_t)\r> +\e^{-2\varK_1(X_{[0,t]})}|\nabla f|^2(X_t).
\end{align*}
By expanding the terms above, we get
\begin{align*}
  |Q_t&\partrinv0t\nabla f(X_t)|^2-\e^{-2\varK_1(X_{[0,t]})}|\nabla f|^2(X_t)\\
      &\leq 4\left(\e^{\frac{1}{2}\left(\varK_2(X_{[0,t]})-\varK_1(X_{[0,t]})\right)}-1\right)|\nabla f|^2+4\l<\nabla f, Q_t\partrinv0t\nabla f(X_t)\r>\\
      &\qquad -4\e^{-\varK_1(X_{[0,t]})}\l<\nabla f, \partrinv0t\nabla f(X_t)\r>.
\end{align*}
We observe that
$|\nabla P_tf|^2\leq \E[|Q_t\partrinv0t\nabla f(X_t)|^2]$ and take
expectation on both sides of the inequality above, to obtain
\begin{align*}
  &|\nabla P_tf|^2-\E\left[\e^{-2\varK_1(X_{[0,t]})}|\nabla f|^2(X_t)\right]\\
  &\leq 4\left(\e^{\frac{1}{2}\left(\varK_2(X_{[0,t]})-\varK_1(X_{[0,t]})\right)}-1\right)|\nabla f|^2+4\E\l<\nabla f, \nabla P_tf-\e^{-\varK_1(X_{[0,t]})}\partrinv0t\nabla f(X_t)\r>.
\end{align*}
Combining this with \eqref{eq2-1} completes the proof of
``$\text{(i)}\Rightarrow\text{(ii)}$''.  \smallskip

``$\text{(i)}\Rightarrow\text{(iii)}$'': It is well known that if
$f\in C_{N}^{\infty}(M)$, then $NP_tf=0$ for $t>0$.  Combined with
It\^o's formula, we obtain
\begin{align*}
  \vd (P_{t-s}f^{2/p})^p(X_s)&=\vd M_s+(L+\partial_s)(P_{t-s}f^{2/p}(X_{s}))^p\,\vd s\notag\\
                             &=\vd M_s+p(p-1)(P_{t-s}f^{2/p}(X_{s}))^{p-2}|\nabla P_{t-s}f^{2/p}|^2(X_s)\,\vd s\notag\\
                             &\quad +p(P_{t-s}f^{2/p})^{p-1}NP_{t-s}f^{2/p}(X_{s})\,\vd l_s\\
                             &=\vd M_s+p(p-1)(P_{t-s}f^{2/p}(X_{s}))^{p-2}|\nabla P_{t-s}f^{2/p}|^2(X_s)\,\vd s
\end{align*}
where $M_s$ is a local martingale. The rest of the argument is then
similar to the proof of Theorem~\ref{th1}; we skip it here.

``(ii) $\Rightarrow$ (i)'':\quad Conversely, for $x\in \mathring{M}$
and $f\in C_{N}^{\infty}(M)$ such that $\Hess_f(x)=0$, we have
\begin{align}\label{eq1-1}
  & \frac{|\nabla P_tf|^2-P_t|\nabla f|^2}{t}+\E\l[\frac{1-\e^{-2\varK_1(X_{[0,t]})}}{t}|\nabla f|^2(X_t)\r]\notag\\
  & \leq 4\bigg(\frac{\E[\e^{\frac{1}{2}\left(\varK_2(X_{[0,t]})-\varK_1(X_{[0,t]})\right)}-1]}{t}|\nabla f|^2
    +\frac{\big<\nabla f, \nabla P_tf-\partrinv0t\nabla f(X_t)\big>}{t}\notag\\
  &\qquad +\l<\nabla f, \E\l[\frac1t \left(1-\e^{-\varK_1(X_{[0,t]})}\right)\partrinv0t\nabla f(X_t)\r]\r>\bigg)\wedge 0.
\end{align}
By Lemma \ref{lem1}(i), there exists $r>0$ such that
$B(x,r)\subseteq \mathring{M}$ and
\begin{align*}
  \lim_{t\rightarrow 0}\frac{1-\e^{-2\varK_1(X_{[0,t]})}}{t}&=\lim_{t\rightarrow 0}\frac{1-\e^{-2\varK_1(X_{[0,t\wedge \sigma_r]})}+{\text{\rm o}(t)}}{t}=\lim_{t\rightarrow 0}\frac{2K_1(x)(t\wedge \sigma_r)+\text{\rm o}(t)}{t}=2K_1(x).
\end{align*}
Similarly, we have
\begin{align*}
  \lim_{t\rightarrow 0}\frac1t \E\left[\e^{\frac{1}{2}(\varK_2(X_{[0,t]})-\varK_1(X_{[0,t]}))}-1\right]=\frac{K_2(x)-K_1(x)}{2},
\end{align*}
and
\begin{align*}
  \lim_{t\rightarrow 0}\bigg<\nabla f, \E\bigg[\frac{(1-\e^{-\varK_1(X_{[0,t]})})}{t}\partrinv0t\nabla f(X_t)\bigg]\bigg>=K_1(x)|\nabla f|^2(x).
\end{align*}
Thus, letting $t\rightarrow 0$ on both sides of \eqref{eq1-1} and
using Lemma \ref{lemma6}, we obtain
$$-2\Ric^Z(\nabla f, \nabla f)+2K_1(x)|\nabla f|^2
\leq \l[2(K_2(x)-K_1(x))|\nabla f|^2-4\Ric^Z(\nabla f, \nabla
f)+4K_1(x)|\nabla f|^2\r]\wedge 0,$$ i.e.,
$$K_1(x)|\nabla f|^2\leq \Ric^Z(\nabla f, \nabla f)\leq K_2(x)|\nabla f|^2.$$
We choose $x\in \partial M$ and $f\in C_{N}^{\infty}(M)$. We can
rewrite the inequality in item (ii) as
\begin{align*}
  & \frac{\sqrt{\pi}\,(|\nabla P_tf|^2-P_t|\nabla f|^2)}{2\sqrt{t}}
    +\E\bigg[\frac{\sqrt{\pi}\,\big(1-\e^{-2\varK_1(X_{[0,t]})}\big)}{2\sqrt{t}}\,|\nabla f|^2(X_t)\bigg]\\
  & \quad\leq 4\bigg[\frac{\sqrt{\pi}\,\E\big[\e^{\frac{1}{2}\left(\varK_2(X_{[0,t]})-\varK_1(X_{[0,t]})\right)}-1\big]}{2\sqrt{t}}|\nabla f|^2+\frac{\sqrt{\pi}\,\big<\nabla f, \nabla P_tf-\partrinv0t\nabla f(X_t)\big>}{2\sqrt{t}}\\
  &\quad\qquad +\bigg<\nabla f, \E\bigg[\frac{\sqrt{\pi}\,(1-\e^{-\varK_1(X_{[0,t]})})}{2\sqrt{t}}\partrinv0t\nabla f(X_t)\bigg]\bigg>\bigg]\wedge 0.
\end{align*}
Now letting $t\rightarrow 0$, by Lemma \ref{II-form} and Lemma
\ref{lem5}, we obtain
\begin{align*}
  -2\II&(\nabla f, \nabla f)+2\sigma_1(x)|\nabla f|^2\\
       &\quad\leq [-4\II(\nabla f, \nabla f)+2(\sigma_2(x)-\sigma_1(x))|\nabla f|^2+4\sigma_1(x)|\nabla f|^2]\wedge 0,
\end{align*}
i.e.,
$$\sigma_1(x)|\nabla f|^2(x)\leq \II(\nabla f, \nabla f)(x)\leq
\sigma_2(x)|\nabla f|^2(x).$$
Similarly, using Lemma \ref{lemma6} and \ref{II-form}, one can prove
``$\text{(iii)}\Rightarrow\text{(i)}$''; we skip the details here.
\end{proof}

\section{Extension to evolving manifolds}\label{Sect:4}
In this section, we deal with the case that the underlying manifold
carries a geometric flow of complete Riemannian metrics. More
precisely, for some $T_c\in (0,\infty]$, we consider the situation of
a $d$-dimensional differentiable manifold $M$ equipped with a $C^1$
family of complete Riemannian metrics $(g_t)_{t\in [0,T_c)}$.  Let
$\nabla^t$ be the Levi-Civita connection and $\Delta_t$ the
Laplace-Beltrami operator associated with the metric $g_t$.  In
addition, let $(Z_t)_{t\in [0,T_c)}$ be a $C^{1}$-family of vector
fields on $M$.  For the sake of brevity, we write
\begin{align*}
  &\mathcal{R}_t^Z(X,Y):=\Ric_t(X,Y)-\l<\nabla^t_XZ_t, Y\r>_t-\frac{1}{2}\partial_tg_t(X,Y),\quad
    X, Y\in T_xM,\ x\in M,
\end{align*}
where $\Ric_t$ is the Ricci curvature tensor with respect to the
metric $g_t$ and $\l<\cdot,\cdot\r>_t:=g_t(\cdot,\cdot)$.

In what follows, for real-valued functions $\phi,\psi$ on
$[0,T_c)\times M$, we write $\psi\leq \mathcal{R}^Z\leq \phi$,
if $$\psi_t|X|^2_t\leq \mathcal{R}_t^Z(X,X)\leq \phi_t |X|^2_t$$ holds
for all $X\in TM$ and $t\in [0,T_c)$, where by definition $|X|_t:=\sqrt{g_t(X,X)}$.  Let
$X_t$ be the diffusion process generated by $L_t:=\Delta_t+Z_t$
(called $L_t$-diffusion) which is assumed to be non-explosive up to
time $T_c$.

We first introduce some notations and recall the construction of
$X_t$. Let $\FM$ be the frame bundle over $M$ and $\OtM$ the
orthonormal frame bundle over $M$ with respect to the metric $g_t$.
We denote by $\pi\colon\FM\rightarrow M$ the projection from
$\FM$ onto $M$.  For $u\in \FM$, let
$$T_{\pi u}M\to T_u\FM,\quad X\mapsto H^{t}_{X}(u),$$
be the $\nabla^{t} $-horizontal lift.  In particular, we consider the
standard-horizontal vector fields $H_{i}^t$ on $\FM$ given by
$$H_{i}^t(u)=H_{ue_i}^t(u), \quad i=1,2,\ldots, d$$
where $\{e_i\}_{i=1}^{d}$ denotes the canonical orthonormal basis
of~$\R^d$.  Let $\{V_{\alpha, \beta}\}_{\alpha,\beta=1}^d$ be the
standard-vertical vector fields on $\FM$,
$$V_{\alpha, \beta}(u):=T\ell_u(\exp(E_{\alpha,\beta})),\quad u\in \FM,$$
where $E_{\alpha,\beta}$ is a basis of the real $d\times d$ matrices,
and $\ell_u\colon\Gl(d;\R)\rightarrow \FM$, $g\mapsto u\cdot g$, is
defined via left multiplication of the general linear group
$\Gl(d;\R)$ on $\FM$.

Let $B_t=(B_t^1,\ldots,B_t^d)$ be a $\R^d$-valued Brownian motion on a
complete filtered probability space
$(\Omega,\{\mathscr{F}_t\}_{t\geq 0}, \P)$.  To construct the
$L_t$-diffusion $X_t$, we first construct the corresponding horizontal
diffusion process $u_t$
by solving the following Stratonovich SDE on $\FM$:
\begin{align}\label{SDE-u}
  \begin{cases}
    \vd u_t=\sqrt{2}\dsum_{i=1}^{d}H_{i}^t(u_t)\circ \vd
    B_t^{i}+H_{Z_t }^t(u_t)\vd t-\frac{1}{2}\dsum_{\alpha,
      \beta=1}^{d}\mathcal{G}_{\alpha,\beta}(t, u_t)V_{\alpha \beta}(u_t)\vd t,\medskip\\
    u_s\in \mathpal{O}_s(M),\ \pi(u_s)=x,\ s\in [0,T_c),
  \end{cases}
\end{align}
where
$\mathcal{G}_{\alpha,\beta}(t,u_t):=\partial_tg_t(u_te_{\alpha},
u_te_{\beta})$.
As explained in \cite{ACT}, the last term is crucial to ensure
$u_t\in\OtM$.  Since $\{H_{Z_{t} }^{t}\}_{t\in [0,T_c)}$ is
$C^{1,\infty}$-smooth, Eq.~\eqref{SDE-u} has a unique solution up to
its lifetime $\zeta:=\dlim_{n\rightarrow \infty}\zeta_n$ where
\begin{align}\label{zeta-n}
  \zeta_n:=\inf\left\{t\in [s,T_c):\rho_t(\pi(u_s), \pi(u_t))\geq n\right\}, \quad n\geq 1,\quad \inf\varnothing:=T_c,
\end{align}
and where $\rho_t$ stands for the Riemannian distance induced by the
metric $g_t$.  Then $X_t^{(s,x)}=\pi(u_t)$ solves the equation
$$\vd X_t^{(s,x)}=\sqrt{2} u_t\circ \vd B_t+Z_t(X_t^{(s,x)})\,\vd t,\quad X_s^{(s,x)}=x:=\pi(u_s)$$
up to the lifetime $\zeta$. By It\^o's formula, for any
$f\in C_0^2(M)$,
$$f(X_t^{(s,x)})-f(x)-\int_s^tL_rf(X^{(s,x)}_r)\vd r=\sqrt{2}\int_s^t\big<\partrinv sr\nabla^rf(X^{(s,x)}_r), u_s^x\vd B_r\bigr>_s,
\quad t\in [s,T_c),$$
is a martingale up to $\zeta$.  In other words, $X_t^{(s,x)}$ is a diffusion
process with generator $L_t$.  In case $s=0$, if there is no risk of
confusion, we write again $X_t^{x}$ instead of~$X_t^{(0,x)}$.

Throughout this section, we assume that the diffusion $X_t$ generated
by $L_t$ is non-explosive up to time $T_c$ (see \cite{KR} for
sufficient conditions ensuring non-explosion).  Then this process
gives rise to an inhomogeneous Markov semigroup
$\{P_{s,t}\}_{0\leq s\leq t< T_c}$ on $\mathcal{B}_b(M)$ by
$$P_{s,t}f(x):=\E\big[f(X_t^{(s,x)})\big]=\E^{(s,x)}\left[f(X_t)\right],\quad x\in M,\ f\in \mathcal{B}_b(M),$$
which is called the diffusion semigroup generated by $L_t$.

We are now in position to present the main result of this section.

\begin{theorem}\label{th3}
  Let $(t,x)\mapsto K_1(t,x)$ and $(t,x)\mapsto K_2(t,x)$ be two
  functions on $M$ such that $K_1\leq K_2$. Suppose that
  \begin{align}\label{eq13}
    \E\left[\e^{-(2+\varepsilon)\int_0^tK_1(s,X_s)\,\vd s}\right]<\infty \quad\text{for some } \varepsilon>0\ \text{ and }\  t\in (0,T_c).
  \end{align}
  The following statements are equivalent to each other:
  \begin{enumerate}
  \item [\rm (i)] the curvature ${\mathcal R}_t^Z$ for time-dependent
    Witten Laplacian satisfies
 $$K_1(t, x)\leq {\mathcal R}_t^Z(x)\leq K_2(t, x), \quad (t,x)\in [0,T_c)\times M;$$
\item [\rm (ii)] for $f\in C_0^{\infty}(M)$ and $0\leq s\leq t<T_c$,
  \begin{align*}
    & |\nabla^sP_{s,t}f|_s^2-\E^{(s,x)}\l[\e^{-2\int_s^t K_1(r,X_r)\,\vd r}\,|\nabla^t f|_t^2(X_t)\r]\\
    & \quad\leq 4\bigg[\left(\E^{(s,x)}\e^{\frac{1}{2}\int_s^t(K_2(r,X_r)-K_1(r,X_r))\,\vd r}-1\right)|\nabla^s f|_s^2+\l<\nabla^s f, \nabla^s P_{s,t}f\r>_s\\
    &\qquad-\l<\nabla^s f,\E^{(s,x)}\l[\e^{-\int_s^t K_1(r,X_r)\,\vd r} \partrinv st\nabla^t f(X_t)\r]\r>_s\bigg]\wedge 0;
  \end{align*}
\item [\rm(ii')] for $f\in C_0^{\infty}(M)$ and $0\leq s\leq t<T_c$,
  \begin{align*}
    &|\nabla^sP_{s,t}f|_s^2-\E^{(s,x)}\l[\e^{-2\int_s^t K_1(r,X_r)\,\vd r}\,|\nabla^t f|_t^2(X_t)\r]\\
    &\quad\leq 4\bigg[\E^{(s,x)}\e^{\frac{1}{2}\int_s^t(K_2(r,X_r)-K_1(r,X_r))\,\vd r}|\nabla^s P_{s,t}f|_s^2\\
    &\qquad-\l<\nabla^s P_{s,t}f, \E^{(s,x)}\l[\e^{-\int_s^t K_1(r,X_r)\,\vd r}\partrinv st\nabla ^t f(X_t)\r]\r>_s\bigg]\wedge 0;
  \end{align*}
\item [\rm(iii)] for $f\in C_0^{\infty}(M)$, $p\in (1,2]$ and
  $0\leq s\leq t<T_c$,
  \begin{align*}
    & \frac{p(P_{s,t}f^2-(P_{s,t}f^{2/p})^p)}{4(p-1)}-\E^{(s,x)}\l[\int_s^t\e^{-2\int_r^t K_1(\tau, X_{\tau})\,\vd \tau}\vd r\times |\nabla^t f|_t^2(X_t)\r]\\
    &\quad\leq 4\int_s^t\l[\E^{(s,x)}\e^{\frac{1}{2}\int_r^t(K_2(\tau, X_{\tau})-K_1(\tau, X_{\tau}))\vd \tau}-1\r]P_{s,r}|\nabla^r f|_r^2\\
    &\qquad +\E^{(s,x)}\l<\nabla^r f(X_r),\nabla^r P_{r,t}f(X_r)-\e^{-\int_r^t K_1(\tau, X_{\tau})\vd \tau}\partrinv rt\nabla^t f(X_t)\r>_r\,\vd r\wedge 0;
  \end{align*}
\item [\rm(iii')] for $f\in C_0^{\infty}(M)$, $p\in (1,2]$ and
  $0\leq s\leq t<T_c$,
  \begin{align*}
    & \frac{p(P_{s,t}f^2-(P_{s,t}f^{2/p})^p)}{4(p-1)}-\E^{(s,x)}\l[\int_s^t\e^{-2\int_r^t K_1(\tau, X_{\tau})\,\vd \tau}\vd r \times |\nabla^t f|_t^2(X_t) \r]\\
    &\quad\leq 4\int_s^t\E^{(s,x)}\e^{\frac{1}{2}\int_r^t(K_2(\tau, X_{\tau})-K_1(\tau, X_{\tau}))\vd \tau}P_{s,r}|\nabla^r P_{r,t}f|_r^2\\
    &\qquad -\E^{(s,x)}\l[\e^{-\int_r^tK_1(\tau, X_{\tau})\vd \tau}\l<\nabla^r P_{r,t}f(X_r),\partrinv rt\nabla^t f(X_t)\r>_r\r]\,\vd r\wedge 0;
  \end{align*}
\item [\rm(iv)] for $f\in C_0^{\infty}(M)$ and $0\leq s\leq t<T_c$,
  \begin{align*}
    & \frac{1}{4}\left(P_{s,t}(f^2\log f^2)-P_{s,t}f^2\log P_{s,t}f^2\right)-\E^{(s,x)}\l[\int_s^t\e^{-2\int_r^t K_1(\tau, X_{\tau})\,\vd \tau}\,\vd r \times |\nabla^t f|_t^2(X_t) \r]\\
    &\quad\leq 4\int_s^t\l[\E^{(s,x)}\e^{\frac{1}{2}\int_r^t(K_2(\tau, X_{\tau})-K_1(\tau, X_{\tau}))\vd \tau}-1\r]P_{s,r}|\nabla^r f|_r^2\\
    &\qquad +\E^{(s,x)}\l<\nabla^r f(X_r),\nabla^r P_{r,t}f(X_r)-\e^{-\int_r^tK_1(\tau, X_{\tau})\vd \tau}\partrinv rt\nabla^t f(X_t)\r>_r\,\vd r\wedge 0;
  \end{align*}
\item [\rm(iv')] for $f\in C_0^{\infty}(M)$ and $0\leq s\leq t<T_c$,
  \begin{align*}
    & \frac{1}{4}(P_{s,t}(f^2\log f^2)-P_{s,t}f^2\log P_{s,t}f^2)-\E^{(s,x)}\l[\int_s^t\e^{-2\int_r^t K_1(\tau, X_{\tau})\,\vd \tau}\,\vd r\times|\nabla^t f|_t^2(X_t) \r]\\
    &\quad\leq 4\int_s^t\E^{(s,x)}\e^{\frac{1}{2}\int_r^t(K_2(\tau, X_{\tau})-K_1(\tau, X_{\tau}))\vd \tau}P_{s,r}|\nabla^{r} P_{r,t}f|_r^2\\
    &\quad \quad -\E^{(s,x)}\l[\e^{-\int_r^tK_1(\tau, X_{\tau})\vd \tau}\l<\nabla^r P_{r,t}f(X_r), \partrinv rt\nabla^t f(X_t)\r>_r\r]\,\vd r\wedge 0.
  \end{align*}
\end{enumerate}
\end{theorem}

\begin{remark}\label{rem5}
  By \cite{Cheng15}, the integral condition \eqref{eq13} can be
  satisfied if $K_2(t,\cdot)/\rho_t^2\rightarrow 0$ as
  $\rho_t^2\rightarrow \infty$ and one of the following conditions
  is satisfied.
  \begin{itemize}
  \item [(A1)] there exists a non-negative continuous function $C$ on
    $[0,T_c)$ such that for all $t\in [0,T_c)$,
$$\mathcal{R}^Z_t\geq -C(t);$$
\item [(A2)] there exist two non-negative continuous functions
  $C_1, C_2$ on $[0,T_c)$ such that for all $t\in [0,T_c)$,
 $$\Ric_t\geq -C_1(t)(1+\rho_t^2)\ \
 {\rm and} \quad \partial_t\rho_t+ \l<Z_t ,\nabla^{t}\rho_t\r>_t\leq
 C_2(t)(1+\rho_t).$$
\end{itemize}
\end{remark}

To prove the theorem, we need the following lemmas: the 
derivative formula and characterization formulae for
${\mathcal R}_t^Z$.  For $s\leq t$, let
${\mathcal R}_{\partr st}^Z:=\partrinv st\circ {\mathcal
  R}_t^Z(X_t)\circ \partr st$.

\begin{lemma}\label{lem3}
  (\cite[Theorem 3.1]{Cheng15}) Let $\mathcal{R}_t^Z(x)\geq K(t,x)$
  for all $t\in [0,T_c)$ and suppose that
$$\sup_{r\in [0,t]}\E\left[\e^{-\int_0^rK(s,X_s)\vd s}\right]<\infty$$
for $t\in [0,T_c)$. Then, for $0\leq s\leq t$,
$$\nabla^s P_{s,t}f(x)=\E^{(s,x)}\left[Q_{s,t}\partrinv st\nabla^t f(X_t^{(s,x)})\right],$$
where for fixed $s\geq0$, the random family $Q_{s,t}\in\Aut(T_{X_s}M)$
is constructed for $t\geq s$ as solution to the equation:
\begin{align}\label{eq-Q3}
  \frac{\vd Q_{s,t}}{\vd t}=-Q_{s,t}\,{\mathcal R}_{\partr st}^Z,\quad Q_{s,s}=\id.
\end{align}
\end{lemma}

\begin{lemma}\label{lem2}
  For $s\in [0,T_c)$ and $x\in M$, let $X\in T_x M$ with
  $|X|_s=1$. Furthermore, let $f\in C_0^{\infty}(M)$ be such that
  $\nabla^{s} f(x)=X$ and $\Hess^s_f(x)=0$, and set $f_n=n+f$ for
  $n\geq 1$. Then,
  \begin{enumerate}[\rm(i)]
  \item for any $p>0$,
    \begin{align*}
      \mathcal{R}^Z_s(X,X)
      &=\lim_{t\downarrow
        s}\frac{P_{s,t}|\nabla^tf|^p_t(x)-|\nabla^{s}P_{s,t}f|^p_s(x)}{p(t-s)};
    \end{align*}
  \item for any $p>1$,
    \begin{align}
      \mathcal{R}^Z_s(X,X)
      &=\lim_{n\rightarrow\infty}\lim_{t\downarrow s
        }\frac{1}{t-s}\l(\frac{p\big(P_{s,t}f_n^2-(P_{s,t}f_n^{2/{p}})^p\big)}{4(p-1)(t-s)}-|\nabla^{s}P_{s,t}f_n|_s^2\r)(x)\nonumber\\
      &=\lim_{n\rightarrow\infty}\lim_{t\downarrow s}\frac{1}{t-s}\l(P_{s,t}|\nabla^tf|^2_t-\frac{p\big(P_{s,t}f_n^2-(P_{s,t}f_n^{2/p})^p\big)}{4(p-1)(t-s)}\r)(x);\label{2-2}
    \end{align}
  \item $\mathcal{R}^Z_s(X,X)$ is equal to each of the following
    limits:
    \begin{align*}
      \quad\mathcal{R}^Z_s(X,X)&=\lim_{n\rightarrow\infty}\lim_{t\downarrow
                                 s}\frac{1}{(t-s)^2}\l\{(P_{s,t}f_n)\big[P_{s,t}(f_n\log
                                 f_n)-(P_{s,t}f_n)\log
                                 P_{s,t}f_n\big]-(t-s)|\nabla^{s}P_{s,t}f|_s^2\r\}(x)\\ 
                               &=\lim_{n\rightarrow\infty}\lim_{t\downarrow
                                 s}\frac{1}{4(t-s)^2}\l\{4(t-s)P_{s,t}|\nabla^tf|^2_t+(P_{s,t}f_n^2)\log
                                 P_{s,t}f_n^2-P_{s,t}{f_n^2\log f_n^2}\r\}(x);
    \end{align*}
  \item $\mathcal{R}^Z_s(X,X)$ can also be calculated via the
    following limits:
    \begin{align*}
      {\mathcal R}^Z_s(X,X)=&\lim_{t\downarrow s}\frac{\l\{\l<\nabla^s f, \E^{(s,x)}\partrinv st\nabla^t f(X_t)\r>_s-\l<\nabla^s f, \nabla^s P_{s,t} f\r>_s\r\}(x)}{t-s}\\
      =&\lim_{t\downarrow s}\frac{\l\{\l<\nabla^s P_{s,t}f,
         \E^{(s,x)}\partrinv st\nabla^t f(X_t)\r>_s-|\nabla^s
         P_{s,t}f|^2_s\r\}(x)}{t-s}.
    \end{align*}
  \end{enumerate}
\end{lemma}

\begin{proof}
  Without loss of generality, we prove (iv) only for $s=0$. For the
  remaining formulae, the reader is refered to \cite{Cheng15}. We have
  \begin{align*}
    &\lim_{t\downarrow 0}\frac{\big<\nabla^0 f,\E \partrinv0t\nabla^t f(X_t)\big>_0-\l<\nabla^0 f, \E Q_t\partrinv0t\nabla^t f(X_t)\r>_0}{t}\\
    &\quad= \lim_{t\downarrow 0}\l<\nabla^0 f,\E \l[\frac{(\id-Q_t)}{t}\partrinv0t\nabla^t f(X_t)\r]\r>_0\\
    &\quad= \lim_{t\downarrow 0}\l<\nabla^0 f,\E \l[\frac1t{\int_0^tQ_s{\mathcal R}^Z_{\partrinv0s}\,\vd s}\,\partrinv0t\nabla^t f(X_t)\r]\r>_0\\
    &\quad= \mathcal{R}_0^{Z}(\nabla^0 f, \nabla^0 f).
  \end{align*}
  Similarly, we have
  \begin{align*}
    &\lim_{t\downarrow 0}\frac{\l<\nabla^0 P_tf,\E \partrinv0t\nabla^t f(X_t)\r>_0-\l<\nabla^0 P_tf, \E Q_t\partrinv0t\nabla^t f(X_t)\r>_0}{t}\\
    &\quad= \lim_{t\downarrow 0}\l<\nabla^0 P_tf,\E \l[\frac{(\id-Q_t)}{t}\partrinv0t \nabla^t f(X_t)\r]\r>_0\\
    &\quad= \lim_{t\downarrow 0}\l<\nabla^0 P_tf,\E \l[\frac1t\int_0^tQ_s\,{\mathcal R}^Z_{\partr0s}\,\vd s\,\partrinv0t \nabla^t f(X_t)\r]\r>_0\\
    &\quad= \mathcal{R}_0^{Z}(\nabla^0 f, \nabla^0 f).\qedhere
  \end{align*}
\end{proof}

\begin{proof}[Proof of Theorem \ref{th3}.]
  We give the proof of the equivalence (i) and (ii),
  resp.~(ii').\smallskip

  ``(i) implies (ii) and (ii')'': By \eqref{eq-Q3}, we know that
  \begin{align*}
    &\left\|\id-\e^{\frac{1}{2}\int_s^t\left(K_1(r,X_r)+K_2(r,X_r)\right)\,\vd r}Q_{s,t}\right\|\\
    &\quad=\l\|\int_s^t\e^{\frac{1}{2}\int_s^r\left(K_1(u,X_u)+K_2(u,X_u)\right)\vd u}\,Q_{s,r}\,\l(\mathcal{R}_{\partr sr}^Z-\frac{K_1(r,X_r)+K_2(r,X_r)}{2}\id\r)\,\vd r\r\|\\
    &\quad\leq \int_s^t\e^{\frac{1}{2}\int_s^r\left(K_1(u,X_u)+K_2(u,X_u)\right)\vd u}\,\|Q_{s,r}\|\,\frac{K_2(r,X_r)-K_1(r,X_r)}{2}\,\vd r \\
    &\quad\leq \int_s^t\e^{\frac{1}{2}\int_s^r\left(K_2(u,X_u)-K_1(u,X_u)\right)\vd u}\,\frac{K_2(r,X_r)-K_1(r,X_r)}{2}\,\vd r\\
    &\quad=\e^{\frac{1}{2}\int_s^t\left(K_2(u,X_u)-K_1(u,X_u)\right)\vd u}-1.
  \end{align*}
  By a similar discussion as in the proof of Theorem \ref{th1}, we
  have
  \begin{align*}
    &\left|2a\nabla^s f+2b\nabla^sP_{s,t}f-Q_{s,t}\,\partrinv st \nabla^tf(X_t)\right|_s^2\\
    &\leq \e^{\frac{1}{2}\int_s^t\left(K_2(r,X_r)-K_1(r,X_r)\right)\vd r}\left|2a\nabla^sf+2b\nabla^sP_{s,t}f-\e^{-\frac{1}{2}\int_s^t(K_1(r,X_r)+K_2(r,X_r))\vd r}\partrinv st\nabla^tf(X_t)\right|_s^2\\
    &\quad +\e^{-\int_s^t\left(K_1(r,X_r)+K_2(r,X_r)\right)\vd r}\left(\e^{\int_s^t\left(K_2(r,X_r)-K_1(r,X_r)\right)\vd r}-\e^{\frac{1}{2}\int_s^t\left(K_2(r,X_r)-K_1(r,X_r)\right)\vd r}\right)|\nabla^tf(X_t)|^2_t\\
    &=\e^{\frac{1}{2}\int_s^t\left(K_2(r,X_r)-K_1(r,X_r)\right)\vd r}\left|2a\nabla^sf+2b\nabla^sP_{s,t}f\right|_s^2\\
    &\quad-2\e^{-\int_s^tK_1(r,X_r)\vd r}\l<2a\nabla^sf+2b\nabla^sP_{s,t}f, \partrinv st\nabla^tf(X_t)\r>_s +\e^{-2\int_s^tK_1(r,X_r)\vd r}|\nabla^tf(X_t)|^2_t
  \end{align*}
  where $a,b$ are constants such that $a+b=1$.  From this, we obtain
  \begin{align}\label{add-5}
    &\E\left|u_sQ_{s,t}u_t^{-1}\nabla^tf(X_t)\right|_s^2-\E\l[\e^{-2\int_s^tK_1(r,X_r)\vd r}|\nabla^tf(X_t)|^2_t\r]\notag\\
    &\leq \left(\E\e^{\frac{1}{2}\int_s^t(K_2(r,X_r)-K_1(r,X_r))\vd r}-1\right)|2a\nabla^sf+2b\nabla^sP_{s,t}f|_s^2\notag\\
    &\quad -2\l<2a\nabla^sf+2b\nabla^sP_{s,t}f, \E\l[\e^{-\int_s^tK_1(r,X_r)\vd r}\partrinv st\nabla^tf(X_t)\r]\r>_s\notag\\
    &\quad +2\l<2a\nabla^sf+2b\nabla^sP_{s,t}f, \nabla^s P_{s,t}f\r>_s.
  \end{align}
  Moreover, by the derivative formula (Lemma \ref{lem3}), we have
$$|\nabla^sP_{s,t}f|_s^2\leq \E|Q_{s,t}\partrinv st \nabla^tf(X_t)|_s^2$$
which combines with \eqref{add-5} implies
\begin{align*}
  &|\nabla^sP_{s,t}f|_s^2-\E\l[\e^{-2\int_s^tK_1(r,X_r)\vd r}|\nabla^tf(X_t)|^2_t\r]\\
  &\leq \left(\E\e^{\frac{1}{2}\int_s^t(K_2(r,X_r)-K_1(r,X_r))\vd r}-1\right)\left|2a\nabla^sf+2b\nabla^sP_{s,t}f\right|_s^2\\
  &\quad -2\l<2a\nabla^sf+2b\nabla^sP_{s,t}f, \E\l[\e^{-\int_s^tK_1(r,X_r)\vd r}\partrinv st\nabla^tf(X_t)\r]\r>_s\\
  &\quad +2\l<2a\nabla^sf+2b\nabla^sP_{s,t}f, \nabla^s P_{s,t}f\r>_s.
\end{align*}
Hence taking $a=1$, $b=0$ and $a=0$, $b=1$ in the above inequalities,
we complete the proof of
``$\text{(i)}\Rightarrow \text{(ii)(ii')}$''.\smallskip

``(i) $\Rightarrow$ (iii)'': \ By It\^{o}'s formula, for
$f\in C_0^{\infty}(M)$,
\begin{align*}
  \vd (P_{s,t}f^{2/p})^p(X_s)&=\vd M_s+(L_s+\partial_s)(P_{s,t}f^{2/p}(X_{s}))^p\,\vd s\notag\\
                             &=\vd M_s+p(p-1)(P_{s,t}f^{2/p}(X_{s}))^{p-2}|\nabla^s P_{s,t}f^{2/p}|_s^2(X_s)\,\vd s\notag\\
                             &=\vd M_s+p(p-1)(P_{s,t}f^{2/p}(X_{s}))^{p-2}|\nabla^s P_{s,t}f^{2/p}|_s^2(X_s)\,\vd s
\end{align*}
where $M_s$ is a local martingale. The rest of the proof then is
similar to the one of Theorem \ref{th1}; we skip the details
here.\smallskip

``$\text{(ii) and (ii')}\Rightarrow\text{(i)}$'':
\begin{align*}
  & \frac{|\nabla^sP_{s,t}f|_s^2-P_{s,t}|\nabla^t f|^2_t}{t-s}+\E^{(s,x)}\bigg[\frac{1-\e^{-2\int_s^t K_1(r,X_r)\vd r}}{t-s}\,|\nabla^t f|_t^2(X_t)\bigg]\\
  & \leq 4\bigg[\frac{\E^{(s,x)}\e^{\frac{1}{2}\int_s^t(K_2(r,X_r)-K_1(r,X_r))\vd r}-1}{t-s}|\nabla^s f|_s^2+\frac{\l<\nabla^s f, \nabla^s P_{s,t}f-\E \partrinv st\nabla^tf(X_t)\r>_s}{t-s}\\
  &\qquad-\bigg<\nabla^s f,\E^{(s,x)}\bigg[\frac{\e^{-\int_s^t K_1(r,X_r)\vd r}-1}{t-s} \partrinv st\nabla^t f(X_t)\bigg]\bigg>_s\bigg]\wedge 0;
\end{align*}
Letting $t\downarrow s$ and using Lemma \ref{lem2} (i) (iv), we have
\begin{align*}
  -2&\mathcal{R}_s^{Z}(\nabla^s f, \nabla^s f)+2K_1(s,x)|\nabla^s f|_s^2\\
    &\leq 4\l[\frac{1}{2}(K_2(s,x)-K_1(s,x))|\nabla^s f|_s^2-\mathcal{R}_s^Z(\nabla^s f, \nabla^s f)+K_1(s,x)|\nabla^s f|^2_s\r]\wedge 0,
\end{align*}
that is
 $$K_1(s,x)|\nabla^s f|_s^2(x)\leq{\mathcal{R}_s^Z(\nabla^sf, \nabla^sf)}(x)\leq K_2(s,x)|\nabla^s f|_s^2(x).$$
 Similarly, (ii') implies (i) as well. We skip the details here.
\end{proof}

Based on our characterizations for pinched curvature on evolving
manifolds, we can define (weak) solutions to some geometric flows.

\begin{corollary}\label{cor1}
  Let $(t,x)\mapsto K(t,x)$ be some function on $[0,T_c)\times M$.
  The following statements are equivalent to each other:
  \begin{enumerate}[\rm(i)]
  \item [\rm(i)] the family $(M, g_t)_{t\in [0,T_c)}$ evolves by
 $$\frac{1}{2}\partial_tg_t=\Ric_t-\nabla^tZ_t-K(t,\cdot)g_t, \quad t\in [0,T_c);$$
\item [\rm(ii)] for $f\in C_0^{\infty}(M)$ and $0\leq s\leq t<T_c$,
  \begin{align*}
    & |\nabla^sP_{s,t}f|_s^2-\E^{(s,x)}\l[\e^{-2\int_s^tK(r,X_r)\vd r}|\nabla^t f|_t^2(X_t)\r]\\
    & \leq 4\bigg[\l<\nabla^s f, \nabla^s P_{s,t}f\r>_s-\l<\nabla^s f,\E^{(s,x)}\l[\e^{-\int_s^tK(r,X_r)\vd r}\partrinv st\nabla^t f(X_t)\r]\r>_s\bigg]\wedge 0;
  \end{align*}
\item [\rm(ii')] for $f\in C_0^{\infty}(M)$ and $0\leq s\leq t<T_c$,
  \begin{align*}
    &|\nabla^sP_{s,t}f|_s^2-\E^{(s,x)}\l[\e^{-2\int_s^tK(r,X_r)\vd r}|\nabla^t f|_t^2(X_t)\r]\\
    &\leq 4\bigg[|\nabla^s P_{s,t}f|_s^2-\l<\nabla^s P_{s,t}f, \E^{(s,x)}\l[\e^{-\int_s^tK(r,X_r)\vd r}\partrinv st\nabla ^t f(X_t)\r]\r>_s\bigg]\wedge 0;
  \end{align*}
\item [\rm(iii)] for $f\in C_0^{\infty}(M)$, $p\in (1,2]$ and
  $0\leq s\leq t<T_c$,
  \begin{align*}
    &\frac{p(P_{s,t}f^2-(P_{s,t}f^{2/p})^p)}{4(p-1)}-\E^{(s,x)}\l[\int_s^t\e^{-2\int_r^tK(\tau,X_{\tau})\vd \tau}\vd r\times|\nabla^t f|_t^2(X_t)\r]\\
    &\quad \leq 4\int_s^t\E^{(s,x)}\l<\nabla^r f(X_r),\nabla^r P_{r,t}f(X_r)-\e^{-\int_r^tK(\tau, X_{\tau})\vd \tau}\partrinv rt\nabla^t f(X_t)\r>_r\vd r\wedge 0;
  \end{align*}
\item [\rm(iii')] for $f\in C_0^{\infty}(M)$, $p\in (1,2]$ and
  $0\leq s\leq t<T_c$,
  \begin{align*}
    & \frac{p(P_{s,t}f^2-(P_{s,t}f^{2/p})^p)}{4(p-1)}-\E^{(s,x)}\l[\int_s^t\e^{-2\int_r^tK(\tau,X_{\tau})\vd \tau}\vd r\times |\nabla^t f|_t^2(X_t)\r]\\
    &\quad\leq 4\int_s^tP_{s,r}|\nabla^r P_{r,t}f|_r^2-\E^{(s,x)}\l<\nabla^r P_{r,t}f(X_r),\e^{-\int_r^tK(\tau, X_{\tau})\vd \tau}\partrinv rt\nabla^t f(X_t)\r>_r\,\vd r\wedge 0;
  \end{align*}
\item [\rm(iv)] for $f\in C_0^{\infty}(M)$ and $0\leq s\leq t<T_c$,
  \begin{align*}
    & \frac{1}{4}\left(P_{s,t}(f^2\log f^2)-P_{s,t}f^2\log P_{s,t}f^2\right)-\E^{(s,x)}\l[\int_s^t\e^{-2\int_r^tK(\tau,X_{\tau})\vd \tau}\vd r\times|\nabla^t f|_t^2(X_t)\r]\\
    & \leq 4\int_s^t\E^{(s,x)}\l<\nabla^r f(X_r),\nabla^r P_{r,t}f(X_r)-\e^{-\int_r^tK(\tau, X_{\tau})\vd \tau}\partrinv rt\nabla^t f(X_t)\r>_r\vd r\wedge 0;
  \end{align*}
\item [\rm(iv')] for $f\in C_0^{\infty}(M)$ and $0\leq s\leq t<T_c$,
  \begin{align*}
    & \frac{1}{4}(P_{s,t}(f^2\log f^2)-P_{s,t}f^2\log P_{s,t}f^2)-\E^{(s,x)}\l[\int_s^t\e^{-2\int_r^tK(\tau,X_{\tau})\vd \tau}\vd r\times|\nabla^t f|_t^2(X_t)\r]\\
    &\leq 4\int_s^t\l[P_{s,r}|\nabla^r P_{r,t}f|_r^2-\E^{(s,x)}\l<\nabla^r P_{r,t}f(X_r),\e^{-\int_r^tK(\tau, X_{\tau})\vd \tau}\partrinv rt\nabla^t f(X_t)\r>_r\r]\,\vd r\wedge 0.
  \end{align*}
\end{enumerate}
\end{corollary}

\begin{remark}\label{rem4}
  In Corollary \ref{cor1}, if $Z_t\equiv 0$ and $K \equiv 0$, the
  results characterize solutions to the Ricci flow, see \cite{HN} for
  functional inequalities on path space characterizing Ricci flow.
\end{remark}

\proof[Acknowledgments]This work has been supported by Fonds National
de la Recherche Luxembourg (Open project O14/7628746 GEOMREV). The
first named author acknowledges support by NSFC (Grant No.~A011002)
and Zhejiang Provincial Natural Science Foundation of China (Grant
No. LQ16A010009).

\bibliographystyle{amsplain}%

\bibliography{CT16a}
\end{document}